\documentclass[11pt]{article}
\usepackage{amssymb}
\usepackage{graphicx}
\usepackage{setspace}
  \usepackage{paralist}
  \usepackage{longtable}
   \usepackage{multirow}
    \usepackage{rotating}
\usepackage{fancyhdr}

\usepackage{amsmath, amsthm, amssymb}
\usepackage{authblk}
\usepackage{graphicx}
\usepackage{float}
\usepackage{hyperref}
\usepackage[margin=1in]{geometry}
\usepackage{comment}
\usepackage{color}
\usepackage[numbers,sort&compress]{natbib}
\allowdisplaybreaks[4]
\numberwithin{equation}{section}

\date{}

\newtheorem{theorem}{Theorem}[section]
\newtheorem{proposition}[theorem]{Proposition}
\newtheorem{lemma}[theorem]{Lemma}

\newtheorem{remark}{Remark}

\newcommand{\be}{\begin{equation}}
\newcommand\ee{\end{equation}}
\newcommand\bes{\begin{eqnarray}}
\newcommand\ees{\end{eqnarray}}
\newcommand\bess{\begin{eqnarray*}}
\newcommand\eess{\end{eqnarray*}}

\title{The global strong solution to a 3D compressible system with fractional viscous term}

\author{Mengqian Liu,\ \ Lei Niu\thanks{Corresponding to: lei.niu@dhu.edu.cn(L. Niu)},\ \ Zhigang Wu\\
Department of Mathematics, Donghua University, Shanghai, P. R. China.}

\begin{document}

\maketitle
\renewcommand{\thefootnote}{\fnsymbol{footnote}}

\maketitle
{{\bf  Abstract:}} In this paper, we study the Cauchy problem to the 3D fractional compressible isentropic generalized Navier-Stokes equations for viscous compressible fluid with one Levy diffusion process. We obtain the existence and uniqueness of global strong solution for small initial data by providing several commutators via the Littlewood-Paley theory. Moreover, we derive $L^2$-decay rate for the highest derivative of the strong solution without decay loss by using a cancellation  of a low-medium-frequency quantity. Our results improve the results provided in [J. Differential Equations 377(2023): 369-417].

\textbf{{\bf Key Words}:}
Fractional Navier-Stokes equations, global well-posedness, optimal decay rate.

\textbf{{\bf MSC2010}:} 35A09; 35B40; 35Q35.

\section{Introduction}
In this paper, we consider the 3D fractional compressible isentropic Navier-Stokes equations which describe the motion of viscous compressible fluid with one Levy diffusion process \cite{Wangs1,Wangs2} given by
\begin{equation} \label{1.1}
\left\{\begin{array}{ll}
\tilde\rho_{t}+\nabla\cdot(\tilde\rho \mathbf{u})=0,\ (x,t)\in\mathbb{R}^3\times(0,\infty)\\
\tilde\rho \mathbf{u}_{t}+\mu(-\Delta)^\alpha \mathbf u+\tilde\rho\mathbf{u}\cdot\nabla\mathbf u+\nabla P(\tilde\rho)=0,\ (x,t)\in\mathbb{R}^3\times(0,\infty),
\end{array}\right.
\end{equation}
where initial data satisfy
\begin{equation}\label{1.2}
(\tilde\rho,\mathbf{u})(0,x)=(\tilde\rho_0,\mathbf{u}_0)(x)\rightarrow(\tilde\rho_\infty,0),\ as\ |x|\rightarrow\infty.
\end{equation}
Here $\tilde\rho$ and $\mathbf u=(u^1,u^2,u^3)$ are the unknown density and velocity respectively, the pressure $P=P(\tilde\rho)$ is given by the power law $P=A\tilde\rho^\gamma$ with constants $\gamma>1$ and $A>0$, and $\mu>0$ denotes the coefficient of viscosity. The fractional Laplace operator $(-\Delta)^\alpha$ is defined through the Fourier transform by
$$\widehat{(-\Delta)^\alpha f}(\xi)=|\xi|^{2\alpha}\hat f(\xi),$$
where $\alpha$ is a positive constant and $\hat{f}$ is the Fourier transform of the function $f$. We write $\Lambda=(-\Delta)^{\frac{1}{2}}$ for notational convenience.

The model \eqref{1.1} is a direct extension of the classical compressible isentropic Navier-Stokes equations when $\alpha=1$, which has been widely studied, see, for example \cite{DUYZ,feireisl1,feireisl4,guo1,HHW1,HLX,HZ0,HZ,J1,J2,LW,MN1,MN2,Li,LX,WZ,Xin,ZZ} and the references therein. In particularly, Matsumura and Nishida \cite{MN1,MN2} first obtained the global existence of small solution in $H^3(\mathbb{R}^3)$, and further got the $L^2$-decay rate of the solution as the heat equation under an additional assumption: the initial perturbation is small in $L^1$. Duan etal. \cite{DUYZ} studied $L^p$ $(2\leq p\leq 6)$ decay rate for the system with external force and they need not the smallness of the initial perturbation in $L^1$-space. By replacing $L^1$-space by $\dot{B}_{1,\infty}^{-s}$ with $s\in[0,1]$, Li and Zhang \cite{Li} achieved a faster $L^2$-decay rate of  the solution. Guo and Wang \cite{guo1} developed a pure energy method to derive the $L^2$-decay rate of the solution and its derivative in $H^l\cap\dot{H}^{-s}$ with $l\geq3$ and $s\in[0,\frac{3}{2})$.

There are also some results of the compressible Navier-Stokes equations in Besov spaces. Based on scaling considerations for this system, a breakthrough was provided by Danchin \cite{D1}. The author established the global existence of strong solutions for the initial data in the vicinity of the equilibrium in $(\dot{B}_{2,1}^{\frac{d}{2}}\cap\dot{B}_{2,1}^{\frac{d}{2}-1})\times\dot{B}_{2,1}^{\frac{d}{2}-1}$ with $d\geq2$. Subsequently, Charve and Danchin \cite{CD1} and Chen, Miao and Zhang \cite{CMZ2} extended Danchin's result to the general $L^p$ critical Besov spaces. Haspot \cite{H} obtained the same results as in \cite{CD1,CMZ2} by using Hoff's viscous effective flux. Later, Chen etal. verified the ill-posedness in \cite{CMZ3}, which means that the critical Besov space in \cite{CD1,CMZ2,H} for the compressible Navier-Stokes equations can be regraded as the largest one in which the system is well-posed. Recently, Peng and Zhai \cite{peng} proved the global existence for $d$-dimensional compressible Navier-Stokes equations without heat conductivity for $d\geq2$ in $L^2$-framework. We refer to the readers \cite{CD2,D3,DH,XX,ZLZ} for more results on critical spaces for the isentropic or non-isentropic compressible Navier-Stokes equations.

Now, we shall go back to the system (\ref{1.1}). It is physically relevant by replacing the standard Laplacian operators since the fractional diffusion operators model the so-called anomalous diffusion. See \cite{Abe,Jara,Mellet} and the references therein for the topic in physics, probability and finance. Some important results on fractional dissipation for the other fluid models were developed in \cite{Ca,Ki}. Recently, the existence and uniqueness of the global solution of system \eqref{1.1} have been proved in \cite{Wangs1,Wangs2} as stated in Lemma \ref{1 1.10}. In order to enclose the energy estimates, they actually need take advantage of the nonlocal operator $D^{m+\alpha}$ with $m=0,1,2,3$ and establish one elaborate spectral theory of one linearized nonlocal operator involved in the fractional dissipation viscosity $\Lambda^{2\alpha}$ for the system, where the eigenvalues and the eigenvectors does depend upon the fractional order derivative exponent $\alpha$.

\begin{lemma}(\cite{Wangs1,Wangs2})\label{1 1.10}
Let $\alpha\in(\frac{1}{2},1]$. There exist constants $C_0>0$ and $\varepsilon_0$ such that if
$$E_0=\|(\rho_0,\mathbf{u}_0)\|_{H^{3\alpha+1}}+\|(\rho_0+\mathbf{u}_0)\|_{L^1}<\varepsilon_0,$$
then the initial value problem \eqref{1.1}-\eqref{1.3} has a unique solution $(\rho,\mathbf{u})$ globally in time, which satisfy
$$\rho(t,x)\in\mathcal{C}^0(0,\infty;H^{3\alpha+1})\cap\mathcal{C}^1(0,\infty;H^{3\alpha}),$$
$$\mathbf{u}(t,x)\in\mathcal{C}^0(0,\infty;H^{3\alpha+1})\cap\mathcal{C}^1(0,\infty;H^{\alpha+1}),$$
and it has the decay rate
$$\|(\rho,\mathbf{u})(t)\|_{H^2}\leq C_0E_0(1+t)^{-\frac{3}{4\alpha}}.$$
\end{lemma}

Due to the appearance of the fractional dissipation in (\ref{1.1}), it is not easy to obtain the global existence in $H^s(\mathbb{R}^d)$ with $d=2,3$ and $s>\frac{d}{2}$ as in \cite{LYZ} for the classical compressible isentropic Navier-Stokes equations. The main reason is the fractional dissipation is weaker than the classical one, which ultimately results that we have to use the $L^\infty$-estimate of the first derivative of the unknowns ($H^s\hookrightarrow L^\infty$ when $s>\frac{d}{2}$). Hence, it is of interest to further relax the requirement of the regularity of the strong solution in this paper. In addition, there are actually some difficulties to further reduce the regularity index to $s+1-\alpha$, such as the nonlinear term
\begin{equation*}
\begin{array}{ll}
\displaystyle\quad\sum\limits_{j\geq 0}2^{2j(s+1-\alpha)}|\langle\Delta_j(\rho{\rm div}\mathbf{u}),\rho_j\rangle|
\lesssim\|\rho{\rm div}\mathbf{u}\|_{\dot{H}^{s+1-\alpha}}\|\rho\|_{\dot{H}^{s+1-\alpha}}
\lesssim\|{\rm div}\mathbf{u}\|_{H^{s+1-\alpha}}\|\rho\|_{H^{s+1-\alpha}}^2,
\end{array}
\end{equation*}
where we need extra 1 order regularity of the velocity $\mathbf{u}$. However, from the system \eqref{1.1}, we know that $\mathbf{u}$ only can achieve extra $\alpha\in(1/2,1)$ order regularity in $L^1$ time campared to initial data. Thus, it is not easy to close the energy argument in $H^{s+1-\alpha}$ space.

The aim of this paper is to further refine the global existence result and the decay rates in \cite{Wangs1,Wangs2}. As for the global existence, we do not need the initial condition in $L^1$-space when deriving the a priori estimate. Moreover, we relax the requirement of the regularity $H^{3\alpha+1}(\mathbb{R}^3)$ in \cite{Wangs2} to $H^{s+1}(\mathbb{R}^3)$ with $s>\frac{3}{2}$. To achieve the new estimate \eqref{1.120}, we first developed a large number of complicated commutators, for instance, (\ref{3.2}), (\ref{3.5}) and (\ref{3.9})-(\ref{3.12}). These commutators on the fractional differential operator eventually help us to derive the priori estimate (\ref{1.120}). We emphasize that these commutators on the fractional differential operator are basic and can also be applied to the other compressible fluid models with fractional dissipation. Inspired by \cite{LYZ}, we shall use the Littlewood-Paley decomposition theory in Sobolev spaces to establish our  energy estimate. In fact, the higher refinement of the low-high decomposition in Littlewood-Paley theorem is also important to relax the requirement of the regularity in this paper. As for the decay rates of the solution, we also extend those results in  \cite{Wangs1,Wangs2}, where they only gave the optimal decay rate for the solution (didn't get the optimal decay rate of the derivatives of the solution). In fact, we can obtain the optimal decay rates for all of the derivatives (especially for the highest order derivatives) of the solution in $H^{s+1}(\mathbb{R}^3)$-framework without the smallness of the initial perturbation in $L^1$.

Specifically, our main results are stated in the following.
\begin{theorem}\label{1 1.1}
Let $\alpha\in (\frac{1}{2},\ 1)$. For any $\rho_0\in H^{s+1}(\mathbb{R}^3),\ \mathbf{u}_0\in H^{s+1}(\mathbb{R}^3)$ with $s>\frac{3}{2}$, there exists a small constant $\eta>0$ such that
\begin{equation}\label{1.3}
\|\rho_0\|_{H^{s+1}}^2+\|\mathbf{u}_0\|_{H^{s+1}}^2\leq\eta,
\end{equation}
so that the Cauchy problem (\ref{1.1})-(\ref{1.2}) has a unique global solution $(\tilde\rho,\ \tilde{\mathbf{u}})$ satisfying
\begin{equation}\label{1.4}
(\rho,\ \mathbf{u})\in C(\mathbb{R}^+;H^{s+1}\times H^{s+1}(\mathbb{R}^3)),\ \ \
(\nabla\rho,\ \Lambda^\alpha \mathbf{u})\in L^2(\mathbb{R}^+;H^s\times H^{s+1}(\mathbb{R}^3))
\end{equation}
and
\begin{equation}\label{1.120}
\|(\rho,{\bf u})(t)\|_{H^{s+1}}^2+\int_0^t\|\nabla\rho(\tau)\|_{H^s}^2+\|\Lambda^\alpha{\bf u}(\tau)\|_{H^{s+1}}^2d\tau\lesssim\|(\rho,{\bf u})(0)\|_{H^{s+1}}^2.
\end{equation}
\end{theorem}
\begin{remark} Actually, the global existence and uniqueness of solution to \eqref{1.1}-\eqref{1.3} in dimension two can be obtained in the same method of this paper.
\end{remark}
%
%

\begin{theorem}\label{1 1.2}
Let $\alpha\in (\frac{1}{2},\ 1)$ and $d=3$. Assume that $\|(\rho_0,\ \mathbf{u}_0)\|_{H^{s+1}}$ is small and $\|(\rho_0,\ \mathbf{u}_0)\|_{L^1}$ is bounded. Then the solution $(\rho,\ \mathbf{u})$ of the Cauchy problem (\ref{1.1})-(\ref{1.2}) satisfies the following optimal decay rate for $\frac{5}{2}<\sigma_0:=s+1<\frac{3+4\alpha}{2}$:
\begin{equation}\label{1.5}
\|\Lambda^\sigma(\rho,\mathbf{u})(t)\|\lesssim(1+t)^{-\frac{3}{4\alpha}-\frac{\sigma}{2}},\ \ 0\leq \sigma\leq \sigma_0.
\end{equation}
\end{theorem}
\begin{remark}We can also get the corresponding decay results for the two dimensional case except the highest order derivative by a similar process in the proof of Theorem \ref{1 1.2}.
\end{remark}

In this paper, we first use the Littlewood-Paley decomposition theory in Sobolev spaces to establish the energy estimate (\ref{1.120}). Then we apply the classical Friedrich's regularization method to build global approximate solutions and prove the existence of a solution by compactness arguments for the small initial data. Moreover, we verify that the solution constructed is unique. Finally, we shall deduce the optimal decay rates for all of the derivatives of the solution by using some ideas in \cite{WW}, where they considered the compressible Navier-Stokes equations with reaction diffusion and got the optimal decay rates for all of the derivatives of the solution by virtue of Fourier theory and a new observation for cancellation of a low-medium-frequency quantity. Here, we extend their result for classical compressible Navier-Stokes equations to the FNS model (\ref{1.1}).

The rest of this paper is organized as follows. In Section 2, the nonlinear problem is reformulated, and several useful lemmas and some known inequalities are given. In Section 3, we obtain the a priori estimate of the solution $(\rho, {\bf u})$ to the system (\ref{2.1})-(\ref{2.2}). Then the main result Theorem \ref{1 1.1} will be proved in Section 4. The optimal decay rates for the highest order derivatives are established in Section 5 based on the frequency decomposition given in the appendix.

\section{Preliminaries}
In this section, the initial problem for (\ref{1.1})-(\ref{1.2}) will be reformulated as follows. Define
$$\kappa=\sqrt{P'(\tilde\rho_\infty)}=\sqrt{A\gamma\tilde\rho_\infty^{\gamma-1}},\ a=\frac{2}{\gamma-1},$$
and
$$\tilde\rho=\frac{\tilde\rho_\infty}{\kappa^a}(\kappa+\frac{1}{a}\rho)^a.$$
Following the ideas in \cite{Wangs1}, we let $\tilde\rho_\infty=1$ and regard $\mu'$ as $\kappa^a\mu$. Then, we may study the equaivalent system given by
\begin{equation}\label{2.1}
\left\{\begin{array}{ll}
\rho_{t}+\mathbf u\cdot\nabla\rho+(\kappa+\frac{1}{a}\rho)\nabla\cdot\mathbf{u}=0,\\
\mathbf{u}_{t}+\frac{\mu'}{(\kappa+\frac{1}{a}\rho)^a}\Lambda^{2\alpha} \mathbf u+\mathbf{u}\cdot\nabla\mathbf u+(\kappa+\frac{1}{a}\rho)\nabla\rho=0.
\end{array}\right.
\end{equation}
The associated initial condition (\ref{1.2}) becomes
\begin{equation} \label{2.2}
(\rho,\mathbf{u})(0,x)=(\rho_0,\mathbf{u}_0)(x),\ x\in\mathbb{R}^d,
\end{equation}
with $\tilde\rho_0=\frac{1}{\kappa^a}(\kappa+\frac{1}{a}\rho_0)^a$.

Throughout this paper, the norms in the Sobolev Spaces $H^s(\mathbb{R}^d)$ are denoted respectively by $\|\cdot\|_{H^s}$. In particular, for $s=0$, we will simply use $\|\cdot\|$ to denote $L^2$-norm and $\|(f,g)\|^2=\|f\|^2+\|g\|^2$. And $H^s$ is the Sobolev space of order $s$ on $\mathbb{R}^d$ with the standard norm
$
\|f\|_{H^s}^2:=\int_{\mathbb{R}^n}(1+|\xi|^2)^s|\hat f(\xi)|^2d\xi<\infty,\ \ \hat{f}=\mathcal{F}(f),
$
where $\mathcal{F}$ denotes Fourier transform. We use $\langle f,g\rangle$ to denote the inner-product in $L^2(\mathbb{R}^d)$. $\Delta_j$ is the non-homogeneous frequency localization operators refer to \cite{BCD}. The symbol $A\lesssim B$ means that there exists a constant $c>0$ independent of $A$ and $B$ such that $A\leq cB$. The symbol $A\approx B$ represents $A\lesssim B$ and $B\lesssim A$. Finally, we denote $D_i=\partial_{x_i}\ (i=1,2,\cdots,d)$, $D^k=\partial_{x_1}^{\alpha_1}\cdots\partial_{x_d}^{\alpha_d}$ with $\alpha_1+\cdots+\alpha_d=k$.

For later use, some Sobolev inequalities are listed as follows.
\begin{lemma}\label{2 2.2}(\cite{LT,CW})
Let $s>0$. Suppose $g\in L^\infty\cap H^s(\mathbb{R})$ and $f\in C^{[s]}(Range(g))$. Then, the composition $f(g(x))\in L^\infty\cap H^s(\mathbb{R})$. Moreover, there exists a constant $C>0$, depending on $s$ and $\|g\|_{L^\infty}$, such that
\begin{equation}\label{2.6}
\|D_x^sf(g(x))\|\leq C\|f\|_{C^{[s]}}\|D^sg\|.
\end{equation}
In particular, for $s\in(0,1]$, one can directly apply the chain rule for fractional derivatives
\begin{equation}\label{2.7}
\|\Lambda_x^sf(g(x))\|\leq C\|Df\|_{L^\infty(\mathbb{R}^d)}\|\Lambda^sg\|,
\end{equation}
where C is a constant depending an $s$ and $\|g\|_{L^\infty}$.
\end{lemma}
\begin{lemma}\label{2 2.3}(\cite{MN2})
Assume that $f(x)$ is a function on $\mathbb{R}^3$.

(i)If $f(x)\in H^s$ with $s>\frac{3}{2}$, then $f\in L^\infty$, and
\begin{equation}\label{2.8}
\|f\|_{L^\infty}\leq C\|f\|_{H^s}.
\end{equation}
(ii)If $f(x)\in H^1$, then $f\in L^p$ for any $p\in[2,6]$ and
\begin{equation}\label{2.9}
\|f\|_{L^p}\leq C\|f\|_{H^1},
\end{equation}
where $C$ is a positive constant.
\end{lemma}
\begin{lemma}\label{2 2.6}(\cite{SS,S})
Let $q>1,2<p<\infty$ and let $\frac{1}{p}+\frac{\alpha}{d}=\frac{1}{q}$. There exists a constant $C>0$ such that if for all $f\in \mathcal{S}'$ is such that $\hat f$ is a function, then
\begin{equation}\label{2.13}
\|f\|_{L^p(\mathbb{R}^d)}\leq C\|\Lambda^\alpha f\|_{L^q(\mathbb{R}^d)}.
\end{equation}
\end{lemma}
We shall also use the following Bernstein inequalities for fractional derivatives.
\begin{lemma}\label{2 2.7}(\cite{Jiu})
Let $\alpha\geq0$. Let $1\leq p\leq q\leq\infty$.

(i) If $f$ satisfies
$$supp\hat f\subset\{\xi\in\mathbb{R}^d:|\xi|\leq K2^j\}$$
for some integer $j$ and a constant $K>0$, then
$$\|(-\Delta)^\alpha f\|_{L^q(\mathbb{R}^d)}\leq C_12^{2j\alpha+jd(\frac{1}{p}-\frac{1}{q})}\|f\|_{L^p(\mathbb{R}^d)}.$$
(ii) If $f$ satisfies
$$supp\hat f\subset\{\xi\in\mathbb{R}^d:K_12^j\leq|\xi|\leq K_22^j\}$$
for some integer $j$ and a constant $0<K_1\leq K_2$, then
$$C_12^{2j\alpha}\|f\|_{L^q(\mathbb{R}^d)}\leq\|(-\Delta)^\alpha f\|_{L^q(\mathbb{R}^d)}\leq C_22^{2j\alpha+jd(\frac{1}{p}-\frac{1}{q})}\|f\|_{L^p(\mathbb{R}^d)},$$
where $C_1$ and $C_2$ are constants depending on $\alpha, p$ and $q$ only.
\end{lemma}

\begin{lemma}\label{2 2.8}(\cite{BCD})
Let $\sigma>0$ and $\sigma_1\in\mathbb{R}$. Then we have, for all $u,\ v\in H^\sigma(\mathbb{R}^d)\cap L^\infty(\mathbb{R}^d)$,
$$\|uv\|_{H^\sigma(\mathbb{R}^d)}\lesssim\|u\|_{H^\sigma(\mathbb{R}^d)}\|v\|_{L^\infty(\mathbb{R}^d)}+\|v\|_{H^\sigma(\mathbb{R}^d)}\|u\|_{L^\infty(\mathbb{R}^d)},$$
$$\|uv\|_{\dot H^\sigma(\mathbb{R}^d)}\lesssim\|u\|_{\dot H^\sigma(\mathbb{R}^d)}\|v\|_{L^\infty(\mathbb{R}^d)}+\|v\|_{\dot H^\sigma(\mathbb{R}^d)}\|u\|_{L^\infty(\mathbb{R}^d)}.$$
Moreover, if $d\geq 2$, then we have, for $u\in H^\sigma(\mathbb{R}^d)\cap H^{\frac{d}{2}-1}(\mathbb{R}^d)$, $v\in H^{\sigma+1}(\mathbb{R}^d)\cap L^\infty(\mathbb{R})$,
$$\|uv\|_{\dot H^\sigma(\mathbb{R}^d)}\lesssim\|u\|_{\dot H^\sigma(\mathbb{R}^d)}\|v\|_{L^\infty(\mathbb{R}^d)}+\|v\|_{\dot H^{\sigma+1}(\mathbb{R}^d)}\|u\|_{H^{\frac{d}{2}-1}(\mathbb{R}^d)}.$$
If $\sigma>\frac{d}{2}$, then $H^\sigma(\mathbb{R})$ embeds into $L^\infty(\mathbb{R}^d)$. Also, for all $u,\ v\in H^\sigma(\mathbb{R}^d)$, it holds that
$$\|uv\|_{H^\sigma(\mathbb{R}^d)}\lesssim\|u\|_{H^\sigma(\mathbb{R}^d)}\|v\|_{H^\sigma(\mathbb{R}^d)}.$$
Otherwise, if $\sigma_1\leq\frac{d}{2}<\sigma$ and $\sigma_1+\sigma>0$, then, for all $u\in H^\sigma(\mathbb{R}^d)$, $v\in H^{\sigma_1}(\mathbb{R}^d)$, it holds that
$$\|uv\|_{H^{\sigma_1}(\mathbb{R}^d)}\lesssim\|u\|_{H^\sigma(\mathbb{R}^d)}\|v\|_{H^{\sigma_1}(\mathbb{R}^d)}.$$
\end{lemma}
\begin{lemma}\label{2 2.9}(\cite{BCD})
Let $\sigma>0$ and let $f$ be a smooth function such that $f(0)=0$. If $u\in H^\sigma(\mathbb{R}^d)$, then there exists a function $C=C(\sigma,f,d)$ such that
\begin{equation*}
\begin{array}{ll}
\|f(u)\|_{H^\sigma(\mathbb{R}^d)}\leq C(\|u\|_{L^\infty(\mathbb{R}^d)})\|u\|_{H^\sigma(\mathbb{R}^d)},\\[2mm]
\|f(u)\|_{\dot{H}^\sigma(\mathbb{R}^d)}\leq C(\|u\|_{L^\infty(\mathbb{R}^d)})\|u\|_{\dot{H}^\sigma(\mathbb{R}^d)}.
\end{array}
\end{equation*}
\end{lemma}
\begin{lemma}\label{2 2.5}(\cite{BCD})
Let $\sigma>\frac{d}{2}$ and let $f$ be a smooth function such that $f'(0)=0$. If $u,v\in H^\sigma(\mathbb{R}^d)$, then there exists a function $C=C(\sigma,f,d)$ such that
\begin{equation*}
\|f(u)-f(v)\|_{H^\sigma(\mathbb{R}^d)}\leq C(\|u\|_{L^\infty(\mathbb{R}^d)},\|v\|_{L^\infty(\mathbb{R}^)})\|u-v\|_{H^\sigma(\mathbb{R}^d)}(\|u\|_{H^\sigma(\mathbb{R}^d)}+\|v\|_{H^\sigma(\mathbb{R}^d)}).
\end{equation*}
\end{lemma}
\begin{lemma}\label{2 2.10}(\cite{BCD})
Let $\sigma>\frac{d}{2}-1$. There exists a positive sequence $\|c_j\|_{l^2}=1$ satisfying $\{c_j\}_{j\geq-1}$ such that
\begin{equation*}
\|[u\cdot\nabla,\Delta_j]f\|_{L^2(\mathbb{R}^d)}\leq Cc_j2^{-j(\sigma+1)}\|\nabla u\|_{H^{\sigma+1}(\mathbb{R}^d)}\|f\|_{H^{\sigma+1}(\mathbb{R}^d)}.
\end{equation*}
\end{lemma}

\section{Priori Estimate}

In this section, we shall derive the priori estimate. Since the two dimensional case is almost the same as the three dimensional case, we only provide the proof of the three dimensional case in this section (we also omit $\mathbb{R}^3$ for simplicity).

In order to simplify the notation, we define the functional set $(\rho, \mathbf{u})\in E(T)$ for
$$(\rho,\mathbf{u})\in C([0,T];H^{s+1}\times H^{s+1}),\ \ (\nabla\rho,\Lambda^\alpha \mathbf{u})\in L^2([0,T];H^s\times H^{s+1}).$$
The corresponding norms are defined as
\begin{equation*}
\begin{array}{ll}
\|(\rho,\mathbf{u})\|_{E(0)}=\|\rho_0\|_{H^{s+1}}^2+\|\mathbf{u}_0\|_{H^{s+1}}^2,\\[2mm]
\|(\rho,\mathbf{u})\|_{E(T)}=\|\rho\|_{L_T^\infty(H^{s+1})}^2+\|\mathbf{u}\|_{L_T^\infty(H^{s+1})}^2+\|\nabla\rho\|_{L_T^2(H^{s})}^2+\|\Lambda^\alpha \mathbf{u}\|_{L_T^2(H^{s+1})}^2.
\end{array}
\end{equation*}
\begin{proposition}\label{3 3.1}
Let $s>\frac{3}{2}$ and let $T>0$. Let $(\rho, \mathbf{u})\in E(T)$ be the solution of Cauchy problem (\ref{2.1}) with initial data $(\rho_0,\mathbf{u}_0)$. Suppose that $\|\rho(t,\cdot)\|_{L^\infty}\leq\frac{1}{2}$. Then we have the following inequality:
\begin{equation*}
\|(\rho,\mathbf{u})\|_{E(T)}\lesssim\|(\rho,\mathbf{u})\|_{E(0)}+\|(\rho,\mathbf{u})\|_{E(T)}^2.
\end{equation*}
\end{proposition}
\begin{proof}
Multiplying $\Delta_j\rho\Delta_j$ and $\Delta_j\mathbf{u}\Delta_j$ on both sides of $(\ref{2.1})_1$ and $(\ref{2.1})_2$ respectively, summing up and integrating over $\mathbb{R}^3$, we have
\begin{equation}\label{3.1}
\begin{array}{ll}
&\displaystyle\frac{1}{2}\frac{d}{dt}(\|\rho_j\|^2+\|\mathbf{u}_j\|^2)+\langle\Delta_j(\frac{\mu'}{(\kappa+\frac{1}{a}\rho)^a}\Lambda^{2\alpha}\mathbf{u}),\mathbf{u}_j\rangle\\[2mm]
=&-\langle\Delta_j(\mathbf{u}\cdot\nabla\rho),\rho_j\rangle-\frac{1}{a}\langle\Delta_j(\rho{\rm div}\mathbf{u}),\rho_j\rangle-\langle\Delta_j(\mathbf{u}\cdot\nabla \mathbf{u}),\mathbf{u}_j\rangle-\frac{1}{a}\langle\Delta_j(\rho\nabla\rho),\mathbf{u}_j\rangle.
\end{array}
\end{equation}
Note that
\begin{equation}\label{3.2}
\begin{array}{ll}
&\langle\Delta_j(\frac{\mu'}{(\kappa+\frac{1}{a}\rho)^a}\Lambda^{2\alpha}\mathbf{u}),\mathbf{u}_j\rangle\\[2mm]
=&\mu\langle\Lambda^\alpha \mathbf{u}_j,\Lambda^\alpha \mathbf{u}_j\rangle+\langle\Delta_j\big((\frac{\mu'}{(\kappa+\frac{1}{a}\rho)^a}-\mu)\Lambda^{2\alpha}\mathbf{u}\big),\mathbf{u}_j\rangle\\[2mm]
=&\mu\langle\Lambda^\alpha \mathbf{u}_j,\Lambda^\alpha \mathbf{u}_j\rangle+\langle[\Delta_j,(\frac{\mu'}{(\kappa+\frac{1}{a}\rho)^a}-\mu)]\Lambda^{2\alpha}\mathbf{u},\mathbf{u}_j\rangle+\langle(\frac{\mu'}{(\kappa+\frac{1}{a}\rho)^a}-\mu)\Lambda^{2\alpha}\mathbf{u}_j,\mathbf{u}_j\rangle\\[2mm]
=&\mu\langle\Lambda^\alpha \mathbf{u}_j,\Lambda^\alpha \mathbf{u}_j\rangle+\langle[\Delta_j,(\frac{\mu'}{(\kappa+\frac{1}{a}\rho)^a}-\mu)]\Lambda^{2\alpha}\mathbf{u},\mathbf{u}_j\rangle+\langle\Lambda^\alpha\big((\frac{\mu'}{(\kappa+\frac{1}{a}\rho)^a}-\mu)\mathbf{u}_j\big),\Lambda^\alpha \mathbf{u}_j\rangle\\[2mm]
=&\langle\frac{\mu'}{(\kappa+\frac{1}{a}\rho)^a}\Lambda^\alpha \mathbf{u}_j,\Lambda^\alpha \mathbf{u}_j\rangle+\langle[\Delta_j,(\frac{\mu'}{(\kappa+\frac{1}{a}\rho)^a}-\mu)]\Lambda^{2\alpha}\mathbf{u},\mathbf{u}_j\rangle+\langle[\Lambda^\alpha,(\frac{\mu'}{(\kappa+\frac{1}{a}\rho)^a}-\mu)]\mathbf{u}_j,\Lambda^\alpha \mathbf{u}_j\rangle,
\end{array}
\end{equation}
so one has
\begin{equation}\label{3.3}
\begin{array}{ll}
&\displaystyle\frac{1}{2}\frac{d}{dt}(\|\rho_j\|^2+\|\mathbf{u}_j\|^2)+\langle\frac{\mu'}{(\kappa+\frac{1}{a}\rho)^a}\Lambda^\alpha \mathbf{u}_j,\Lambda^\alpha \mathbf{u}_j\rangle\\[2mm]
=&-\langle\Delta_j(\mathbf{u}\cdot\nabla\rho),\rho_j\rangle-\frac{1}{a}\langle\Delta_j(\rho{\rm div}\mathbf{u}),\rho_j\rangle-\langle\Delta_j(\mathbf{u}\cdot\nabla \mathbf{u}),\mathbf{u}_j\rangle-\frac{1}{a}\langle\Delta_j(\rho\nabla\rho),\mathbf{u}_j\rangle\\[2mm]
&-\langle[\Delta_j,(\frac{\mu'}{(\kappa+\frac{1}{a}\rho)^a}-\mu)]\Lambda^{2\alpha}\mathbf{u},\mathbf{u}_j\rangle-\langle[\Lambda^\alpha,(\frac{\mu'}{(\kappa+\frac{1}{a}\rho)^a}-\mu)]\mathbf{u}_j,\Lambda^\alpha \mathbf{u}_j\rangle.
\end{array}
\end{equation}
Next, we apply $D_i\Delta_j\rho D_i\Delta_j$ and $D_i\Delta_j\mathbf{u}D_i\Delta_j$ to $(\ref{2.1})_1$ and $(\ref{2.1})_2$ respectively and integrating over $\mathbb{R}^3$, we get
\begin{equation}\label{3.4}
\begin{array}{ll}
&\displaystyle\frac{1}{2}\frac{d}{dt}(\|D_i\rho_j\|^2+\|D_i\mathbf{u}_j\|^2)+\langle D_i\Delta_j(\frac{\mu'}{(\kappa+\frac{1}{a}\rho)^a}\Lambda^{2\alpha}\mathbf{u}),D_i\mathbf{u}_j\rangle\\[2.6mm]
=&-\langle D_i\Delta_j(\mathbf{u}\cdot\nabla\rho),D_i\rho_j\rangle-\frac{1}{a}\langle D_i\Delta_j(\rho{\rm div}\mathbf{u}),D_i\rho_j\rangle-\langle D_i\Delta_j(\mathbf{u}\cdot\nabla \mathbf{u}),D_i\mathbf{u}_j\rangle\\[2.6mm]
&-\frac{1}{a}\langle D_i\Delta_j(\rho\nabla\rho),D_i\mathbf{u}_j\rangle.
\end{array}
\end{equation}
Similar to (\ref{3.2}), we have
\begin{equation}\label{3.5}
\begin{array}{ll}
&\langle D_i\Delta_j(\frac{\mu'}{(\kappa+\frac{1}{a}\rho)^a}\Lambda^{2\alpha}u),D_i\mathbf{u}_j\rangle\\[2mm]
=&\langle\frac{\mu'}{(\kappa+\frac{1}{a}\rho)^a}D_i\Lambda^\alpha u_j,D_i\Lambda^\alpha u_j\rangle+\langle\Delta_j\big(D_i(\frac{\mu'}{(\kappa+\frac{1}{a}\rho)^a}-\mu)\Lambda^{2\alpha}\mathbf{u}\big),D_i\mathbf{u}_j\rangle\\[2mm]
&+\langle[\Delta_j,(\frac{\mu'}{(\kappa+\frac{1}{a}\rho)^a}-\mu)]D_i\Lambda^{2\alpha}\mathbf{u},D_i\mathbf{u}_j\rangle+\langle[\Lambda^\alpha,\frac{\mu'}{(\kappa+\frac{1}{a}\rho)^a}-\mu]D_i\mathbf{u}_j,D_i\Lambda^\alpha \mathbf{u}_j\rangle,
\end{array}
\end{equation}
thus
\begin{equation}\label{3.6}
\begin{array}{ll}
&\!\!\!\!\displaystyle\frac{1}{2}\frac{d}{dt}(\|D_i\rho_j\|^2+\|D_i\mathbf{u}_j\|^2)+\langle\frac{\mu'}{(\kappa+\frac{1}{a}\rho)^a}D_i\Lambda^\alpha \mathbf{u}_j,D_i\Lambda^\alpha \mathbf{u}_j\rangle\\[2mm]
=&\!\!\!\!-\langle D_i\Delta_j(\mathbf{u}\cdot\nabla\rho),D_i\rho_j\rangle-\frac{1}{a}\langle D_i\Delta_j(\rho{\rm div}\mathbf{u}),D_i\rho_j\rangle-\langle D_i\Delta_j(\mathbf{u}\cdot\nabla \mathbf{u}),D_i\mathbf{u}_j\rangle\\[2mm]
&\!\!\!\!-\frac{1}{a}\langle D_i\Delta_j(\rho\nabla\rho),D_i\mathbf{u}_j\rangle-\langle\Delta_j\big(D_i(\frac{\mu'}{(\kappa+\frac{1}{a}\rho)^a}-\mu)\Lambda^{2\alpha}\mathbf{u}\big),D_i\mathbf{u}_j\rangle\\[2mm]
&\!\!\!\!-\langle[\Delta_j,(\frac{\mu'}{(\kappa+\frac{1}{a}\rho)^a}-\mu)]D_i\Lambda^{2\alpha}\mathbf{u},D_i\mathbf{u}_j\rangle-\langle[\Lambda^\alpha,\frac{\mu'}{(\kappa+\frac{1}{a}\rho)^a}-\mu]D_i\mathbf{u}_j,D_i\Lambda^\alpha \mathbf{u}_j\rangle.
\end{array}
\end{equation}
Summing up $\sum\limits_{1\leq i\leq 3}(\ref{3.6})$, we obtain
\begin{equation}\label{3.7}
\begin{array}{ll}
&\!\!\!\!\displaystyle\frac{1}{2}\frac{d}{dt}(\sum\limits_{i=1}^3\|D_i\rho_j\|^2+\sum\limits_{i=1}^3\|D_i\mathbf{u}_j\|^2)+\sum\limits_{i=1}^3\langle\frac{\mu'}{(\kappa+\frac{1}{a}\rho)^a}D_i\Lambda^\alpha \mathbf{u}_j,D_i\Lambda^\alpha \mathbf{u}_j\rangle\\[2mm]
=&\displaystyle\!\!\!\!-\sum\limits_{i=1}^3\langle D_i\Delta_j(\mathbf{u}\cdot\nabla\rho),D_i\rho_j\rangle-\sum\limits_{i=1}^3\frac{1}{a}\langle D_i\Delta_j(\rho{\rm div}\mathbf{u}),D_i\rho_j\rangle-\sum\limits_{i=1}^3\langle D_i\Delta_j(\mathbf{u}\cdot\nabla \mathbf{u}),D_i\mathbf{u}_j\rangle\\[2mm]
&\!\!\!\!\displaystyle-\sum\limits_{i=1}^3\frac{1}{a}\langle D_i\Delta_j(\rho\nabla\rho),D_i\mathbf{u}_j\rangle-\sum\limits_{i=1}^3\langle\Delta_j\big(D_i(\frac{\mu'}{(\kappa+\frac{1}{a}\rho)^a}-\mu)\Lambda^{2\alpha}\mathbf{u}\big),D_i\mathbf{u}_j\rangle\\[2mm]
&\!\!\!\!\displaystyle-\sum\limits_{i=1}^3\langle[\Delta_j,(\frac{\mu'}{(\kappa+\frac{1}{a}\rho)^a}-\mu)]D_i\Lambda^{2\alpha}\mathbf{u},D_i\mathbf{u}_j\rangle-\sum\limits_{i=1}^3\langle[\Lambda^\alpha,\frac{\mu'}{(\kappa+\frac{1}{a}\rho)^a}-\mu]D_i\mathbf{u}_j,D_i\Lambda^\alpha \mathbf{u}_j\rangle.
\end{array}
\end{equation}
Multiplying $\Lambda^2\Delta_j(\ref{2.1})_1$ and $\Lambda^2\Delta_j(\ref{2.1})_2$ by $\Lambda^2\Delta_j\rho$ and $\Lambda^2\Delta_j\mathbf{u}$, respectively, and integrating with respect to $x$ over $\mathbb{R}^3$, we have
\begin{equation}\label{3.8}
\begin{array}{ll}
&\!\!\!\!\displaystyle\frac{1}{2}\frac{d}{dt}(\|\Lambda^2\rho_j\|^2+\|\Lambda^2\mathbf{u}_j\|^2)+\langle\Lambda^2\Delta_j(\frac{\mu'}{(\kappa+\frac{1}{a}\rho)^a}\Lambda^{2\alpha}\mathbf{u}),\Lambda^2 \mathbf{u}_j\rangle\\[2mm]
=&\!\!\!\!-\langle\Lambda^2\Delta_j(\mathbf{u}\cdot\nabla\rho),\Lambda^2\rho_j\rangle-\frac{1}{a}\langle\Lambda^2\Delta_j(\rho{\rm div}\mathbf{u}),\Lambda^2\rho_j\rangle-\langle\Lambda^2\Delta_j(\mathbf{u}\cdot\nabla \mathbf{u}),\Lambda^2\mathbf{u}_j\rangle\\[2mm]
&\!\!\!\!-\frac{1}{a}\langle\Lambda^2\Delta_j(\rho\nabla\rho),\Lambda^2\mathbf{u}_j\rangle.
\end{array}
\end{equation}
Note that $\Lambda^2=-\Delta$, so we have the following relations by integration by parts:
\begin{equation}\label{3.9}
\begin{array}{ll}
&\displaystyle\!\!\!\!\langle\Lambda^2\Delta_j(\frac{\mu'}{(\kappa+\frac{1}{a}\rho)^a}\Lambda^{2\alpha}\mathbf{u}),\Lambda^2\mathbf{u}_j\rangle\\[2mm]
=&\displaystyle\!\!\!\!\mu\langle\Lambda^{2+\alpha}\mathbf{u}_j,\Lambda^{2+\alpha}\mathbf{u}_j\rangle-\langle{\rm div}\Delta_j\big(\nabla(\frac{\mu'}{(\kappa+\frac{1}{a}\rho)^a}-\mu)\Lambda^{2\alpha} \mathbf{u}\big),\Lambda^2\mathbf{u}_j\rangle\\[2mm]
&\displaystyle\!\!\!\!-\langle{\rm div}\Delta_j\big((\frac{\mu'}{(\kappa+\frac{1}{a}\rho)^a}-\mu)\nabla\Lambda^{2\alpha}\mathbf{u}\big),\Lambda^2\mathbf{u}_j\rangle\\[2mm]
=&\displaystyle\!\!\!\!\mu\langle\Lambda^{2+\alpha}\mathbf{u}_j,\Lambda^{2+\alpha}\mathbf{u}_j\rangle-\langle{\rm div}\Delta_j\big(\nabla(\frac{\mu'}{(\kappa+\frac{1}{a}\rho)^a}-\mu)\Lambda^{2\alpha} u\big),\Lambda^2 \mathbf{u}_j\rangle\\[2mm]
&\displaystyle\!\!\!\!+\langle[\Delta_j,(\frac{\mu'}{(\kappa+\frac{1}{a}\rho)^a}-\mu)]\nabla\Lambda^{2\alpha}\mathbf{u},\nabla\Lambda^2\mathbf{u}_j\rangle+\langle(\frac{\mu'}{(\kappa+\frac{1}{a}\rho)^a}-\mu)\nabla\Lambda^{2\alpha}\mathbf{u}_j,\nabla\Lambda^2\mathbf{u}_j\rangle\\[2mm]
=&\!\!\!\!\mu\|\Lambda^{2+\alpha}\mathbf{u}_j\|^2\!\!+\!\!\langle(\frac{\mu'}{(\kappa+\frac{1}{a}\rho)^a}-\mu)\nabla\Lambda^{1+\alpha}\mathbf{u}_j,\nabla\Lambda^{1+\alpha}\mathbf{u}_j\rangle
\!\!-\!\!\langle{\rm div}\Delta_j\big(\nabla(\frac{\mu'}{(\kappa+\frac{1}{a}\rho)^a}-\mu)\Lambda^{2\alpha} \mathbf{u}\big),\Lambda^2 \mathbf{u}_j\rangle\\
\end{array}
\end{equation}
Similarly, we get
\begin{equation}\label{3.10}
\begin{array}{ll}
&\displaystyle\!\!\!\!-\langle\Lambda^2\Delta_j(\mathbf{u}\cdot\nabla\rho),\Lambda^2\rho_j\rangle\\[2mm]
=&\displaystyle\!\!\!\!\langle{\rm div}\Delta_j(\nabla \mathbf{u}\cdot\nabla\rho),\Lambda^2\rho_j\rangle-\langle [\Delta_j,\mathbf{u}\cdot\nabla]\nabla\rho,\nabla\Lambda^2\rho_j\rangle-\langle \mathbf{u}\cdot\nabla\nabla\rho_j,\nabla\Lambda^2\rho_j\rangle\\[2mm]
=&\displaystyle\!\!\!\!\langle{\rm div}\Delta_j(\nabla \mathbf{u}\cdot\nabla\rho),\Lambda^2\rho_j\rangle\!-\!\langle [\Delta_j,\mathbf{u}\cdot\nabla]\nabla\rho,\nabla\Lambda^2\rho_j\rangle
\!+\!\langle\nabla \mathbf{u}\nabla^2\rho_j,\Lambda^2\rho_j\rangle+\frac{1}{2}\langle {\rm div}\mathbf{u},|\Lambda^2\rho_j|^2\rangle,
\end{array}
\end{equation}
\begin{equation}\label{3.11}
\begin{array}{ll}
&\!\!\!\!-\frac{1}{a}\langle\Lambda^2\Delta_j(\rho{\rm div}\mathbf{u}),\Lambda^2\rho_j\rangle-\frac{1}{a}\langle\Lambda^2\Delta_j(\rho\nabla\rho),\Lambda^2\mathbf{u}_j\rangle\\[2.5mm]
=&\!\!\!\!\frac{1}{a}\langle {\rm div}\Delta_j(\nabla\rho{\rm div}\mathbf{u}),\Lambda^2\rho_j\rangle-\frac{1}{a}\langle [\Delta_j,\rho]\nabla{\rm div}\mathbf{u},\nabla\Lambda^2 \rho_j\rangle-\frac{1}{a}\langle\rho \nabla{\rm div}\mathbf{u}_j,\nabla\Lambda^2\rho_j\rangle\\[2.5mm]
&\!\!\!\!+\frac{1}{a}\langle {\rm div}\Delta_j(\nabla\rho\nabla\rho),\Lambda^2\mathbf{u}_j\rangle-\frac{1}{a}\langle [\Delta_j,\rho] \nabla^2\rho,\nabla\Lambda^2\mathbf{u}_j\rangle-\frac{1}{a}\langle\rho \nabla^2\rho_j,\nabla\Lambda^2\mathbf{u}_j\rangle\\[2.5mm]
=&\!\!\!\!\frac{1}{a}\langle {\rm div}\Delta_j(\nabla\rho{\rm div}\mathbf{u}),\Lambda^2\rho_j\rangle-\frac{1}{a}\langle [\Delta_j,\rho]\nabla{\rm div}\mathbf{u},\nabla\Lambda^2 \rho_j\rangle+\frac{1}{a}\langle {\rm div}\Delta_j(\nabla\rho\nabla\rho),\Lambda^2\mathbf{u}_j\rangle\\[2.5mm]
&\!\!\!\!-\frac{1}{a}\langle [\Delta_j,\rho] \nabla^2\rho,\nabla\Lambda^2\mathbf{u}_j\rangle\!+\!\frac{1}{a}\langle\nabla\rho \nabla^2\rho_j,\Lambda^2 \mathbf{u}_j\rangle\!+\!\frac{1}{a}\langle\nabla\rho \Lambda^2\rho_j,\Lambda^2 \mathbf{u}_j\rangle\!+\!\frac{1}{a}\langle\nabla\rho \nabla{\rm div}\mathbf{u}_j,\Lambda^2 \rho_j\rangle,
\end{array}
\end{equation}
and
\begin{equation}\label{3.12}
\begin{array}{ll}
&\!\!\!\!-\langle\Lambda^2\Delta_j(\mathbf{u}\!\cdot\!\nabla \mathbf{u}),\Lambda^2\mathbf{u}_j\rangle
\\[2mm]
=&\!\!\!\!\langle{\rm div}\Delta_j(\nabla \mathbf{u}\!\cdot\!\nabla \mathbf{u}),\Lambda^2\mathbf{u}_j\rangle-\langle [\Delta_j,\mathbf{u}\!\cdot\!\nabla]\nabla \mathbf{u},\nabla\Lambda^2\mathbf{u}_j\rangle-\langle \mathbf{u}\!\cdot\!\nabla\nabla \mathbf{u}_j,\nabla\Lambda^2\mathbf{u}_j\rangle\\[2mm]
=&\!\!\!\!\langle{\rm div}\Delta_j(\nabla \mathbf{u}\!\cdot\!\nabla \mathbf{u}),\Lambda^2\mathbf{u}_j\rangle\!-\!\langle [\Delta_j,\mathbf{u}\!\cdot\!\nabla]\nabla \mathbf{u},\nabla\Lambda^2\mathbf{u}_j\rangle
\!+\!\langle\nabla \mathbf{u}\nabla^2\mathbf{u}_j,\Lambda^2 \mathbf{u}_j\rangle\!+\!\frac{1}{2}\langle {\rm div}\mathbf{u},|\Lambda^2 \mathbf{u}_j|^2\rangle.
\end{array}
\end{equation}
Inserting (\ref{3.9})-(\ref{3.12}) into (\ref{3.8}), we get
\begin{equation}\label{3.13}
\begin{array}{ll}
&\displaystyle\!\!\!\!\frac{1}{2}\frac{d}{dt}(\|\Lambda^2\rho_j\|^2+\|\Lambda^2\mathbf{u}_j\|^2)+\mu\|\Lambda^{2+\alpha}\mathbf{u}_j\|^2+\langle(\frac{\mu'}{(\kappa+\frac{1}{a}\rho)^a}-\mu)\nabla\Lambda^{1+\alpha}\mathbf{u}_j,\nabla\Lambda^{1+\alpha}\mathbf{u}_j\rangle\\[2mm]
=&\displaystyle\!\!\!\!\langle{\rm div}\Delta_j(\nabla \mathbf{u}\cdot\nabla\rho),\Lambda^2\rho_j\rangle-\langle [\Delta_j,\mathbf{u}\cdot\nabla]\nabla\rho,\nabla\Lambda^2\rho_j\rangle+\langle\nabla \mathbf{u}\nabla^2\rho_j,\Lambda^2\rho_j\rangle\\[2.5mm]
&\displaystyle\!\!\!\!+\frac{1}{2}\langle {\rm div}\mathbf{u},|\Lambda^2\rho_j|^2\rangle+\langle{\rm div}\Delta_j(\nabla \mathbf{u}\cdot\nabla \mathbf{u}),\Lambda^2\mathbf{u}_j\rangle-\langle [\Delta_j,\mathbf{u}\cdot\nabla]\nabla \mathbf{u},\nabla\Lambda^2\mathbf{u}_j\rangle\\[2.5mm]
&\displaystyle\!\!\!\!+\langle\nabla \mathbf{u}\nabla^2\mathbf{u}_j,\Lambda^2\mathbf{u}_j\rangle+\frac{1}{2}\langle {\rm div}\mathbf{u},|\Lambda^2\mathbf{u}_j|^2\rangle+\frac{1}{a}\langle {\rm div}\Delta_j(\nabla\rho{\rm div}\mathbf{u}),\Lambda^2\rho_j\rangle\\[2.5mm]
&\displaystyle\!\!\!\!-\frac{1}{a}\langle [\Delta_j,\rho]\nabla{\rm div}\mathbf{u},\nabla\Lambda^2 \rho_j\rangle+\frac{1}{a}\langle {\rm div}\Delta_j(\nabla\rho\nabla\rho),\Lambda^2\mathbf{u}_j\rangle-\frac{1}{a}\langle [\Delta_j,\rho] \nabla^2\rho,\nabla\Lambda^2\mathbf{u}_j\rangle\\[2.5mm]
&\displaystyle\!\!\!\!+\frac{1}{a}\langle\nabla\rho \nabla^2\rho_j,\Lambda^2 \mathbf{u}_j\rangle+\frac{1}{a}\langle\nabla\rho \Lambda^2\rho_j,\Lambda^2 \mathbf{u}_j\rangle+\frac{1}{a}\langle\nabla\rho \nabla{\rm div}\mathbf{u}_j,\Lambda^2\rho_j\rangle\\[2.5mm]
&\displaystyle\!\!\!\!+\langle{\rm div}\Delta_j(\nabla(\frac{\mu'}{(\kappa+\frac{1}{a}\rho)^a}-\mu)\Lambda^{2\alpha} \mathbf{u}),\Lambda^2 \mathbf{u}_j\rangle-\langle[\Delta_j,(\frac{\mu'}{(\kappa+\frac{1}{a}\rho)^a}-\mu)]\nabla\Lambda^{2\alpha}\mathbf{u},\nabla\Lambda^2\mathbf{u}_j\rangle\\
\end{array}
\end{equation}
From the equations $(\ref{2.1})_1$-$(\ref{2.1})_2$ and integration by parts, we have
\begin{equation}\label{3.14}
\begin{array}{ll}
&\displaystyle\!\!\!\!\frac{d}{dt}\langle\nabla\rho_j,\mathbf{u}_j\rangle+\kappa\|\nabla\rho_j\|^2\\[2mm]
=&\!\!\!\!\kappa\|{\rm div}\mathbf{u}\|^2-\langle \nabla\Delta_j(\mathbf{u}\cdot\nabla\rho),\mathbf{u}_j\rangle-\frac{1}{a}\langle\nabla\Delta_j(\rho{\rm div}\mathbf{u}),\mathbf{u}_j\rangle-\frac{1}{a}\langle\Delta_j(\rho\nabla\rho),\nabla\rho_j\rangle\\[2mm]
&\!\!\!\!-\langle\Delta_j(\mathbf{u}\cdot\nabla \mathbf{u}),\nabla\rho_j\rangle-\langle\Delta_j\big((\frac{\mu'}{(\kappa+\frac{1}{a}\rho)^a}-\mu)\Lambda^{2\alpha}\mathbf{u}\big),\nabla\rho_j\rangle-\mu\langle\nabla\rho_j,\Lambda^{2\alpha}\mathbf{u}_j\rangle,
\end{array}
\end{equation}
and
\begin{equation}\label{3.15}
\begin{array}{ll}
&\displaystyle\!\!\!\!\frac{d}{dt}\langle D_i\nabla\rho_j,D_i\mathbf{u}_j\rangle+\kappa\|D_i\nabla\rho_j\|^2\\[2mm]
=&\!\!\!\!\kappa\|D_i{\rm div}\mathbf{u}\|^2\!\!-\!\langle D_i\nabla\Delta_j(\mathbf{u}\!\cdot\!\nabla\rho),D_i\mathbf{u}_j\rangle\!\!-\!\frac{1}{a}\langle D_i\nabla\Delta_j(\rho{\rm div}\mathbf{u}),D_i\mathbf{u}_j\rangle\!\!-\!\frac{1}{a}\langle D_i\Delta_j(\rho\nabla\rho),D_i\nabla\rho_j\rangle\\[2mm]
&\!\!\!\!-\langle\Delta_jD_i(\mathbf{u}\cdot\nabla \mathbf{u}),D_i\nabla\rho_j\rangle\!\!-\!\langle\Delta_jD_i\big((\frac{\mu'}{(\kappa+\frac{1}{a}\rho)^a}-\mu)\Lambda^{2\alpha}\mathbf{u}\big),D_i\nabla\rho_j\rangle\!\!-\!\mu\langle D_i\nabla\rho_j,D_i\Lambda^{2\alpha}\mathbf{u}_j\rangle.
\end{array}
\end{equation}
Summing up $1\leq i\leq3$ in (\ref{3.15}), we have
\begin{equation}\label{3.16}
\begin{array}{ll}
&\displaystyle\!\!\!\!\frac{d}{dt}\sum\limits_{i=1}^3\langle D_i\nabla\rho_j,D_i\mathbf{u}_j\rangle+\kappa\sum\limits_{i=1}^3\|D_i\nabla\rho_j\|^2\\[2mm]
=&\displaystyle\!\!\!\!\kappa\sum\limits_{i=1}^3\|D_i{\rm div}\mathbf{u}\|^2+\sum\limits_{i=1}^3\langle D_i\Delta_j(\mathbf{u}\cdot\nabla\rho),{\rm div}D_i\mathbf{u}_j\rangle+\frac{1}{a}\sum\limits_{i=1}^3\langle D_i\Delta_j(\rho{\rm div}\mathbf{u}),{\rm div}D_i\mathbf{u}_j\rangle\\[2mm]
&\!\!\!\!\displaystyle-\sum\limits_{i=1}^3\frac{1}{a}\langle D_i\Delta_j(\rho\nabla\rho),D_i\nabla\rho_j\rangle-\sum\limits_{i=1}^3\langle\Delta_jD_i(\mathbf{u}\cdot\nabla \mathbf{u}),D_i\nabla\rho_j\rangle\\[2mm]
&\displaystyle\!\!\!\!-\sum\limits_{i=1}^3\langle\Delta_jD_i\big((\frac{\mu'}{(\kappa+\frac{1}{a}\rho)^a}-\mu)\Lambda^{2\alpha}\mathbf{u}\big),D_i\nabla\rho_j\rangle-\mu\sum\limits_{i=1}^3\langle D_i\nabla\rho_j,D_i\Lambda^{2\alpha}\mathbf{u}_j\rangle.
\end{array}
\end{equation}
Next, summing up (\ref{3.3}), (\ref{3.7}), (\ref{3.13}), $\beta_1\times(\ref{3.14})$ and $\beta_2\times(\ref{3.16})$, one has
\begin{equation}\label{3.17}
\begin{array}{ll}
\!\!\!\!&\displaystyle\!\!\!\!\frac{1}{2}\frac{d}{dt}(\|\rho_j\|^2+\|\mathbf{u}_j\|^2+\sum\limits_{i=1}^3\|D_i\rho_j\|^2+\sum\limits_{i=1}^3\|D_i\mathbf{u}_j\|^2+\|\Lambda^2\rho_j\|^2+\|\Lambda ^2\mathbf{u}_j\|^2+2\beta_1\langle\nabla\rho_j,\mathbf{u}_j\rangle\\[2mm]
\!\!\!\!&\displaystyle\!\!\!\!+2\beta_2\sum\limits_{i=1}^3\langle D_i\nabla\rho_j,D_i\mathbf{u}_j\rangle)+\langle\frac{\mu'}{(\kappa+\frac{1}{a}\rho)^a}\Lambda^\alpha \mathbf{u}_j,\Lambda^\alpha \mathbf{u}_j\rangle+\sum\limits_{i=1}^3\langle\frac{\mu'}{(\kappa+\frac{1}{a}\rho)^a}D_i\Lambda^\alpha \mathbf{u}_j,D_i\Lambda^\alpha \mathbf{u}_j\rangle\\[2mm]
\!\!\!\!&\displaystyle\!\!\!\!+\langle(\frac{\mu'}{(\kappa+\frac{1}{a}\rho)^a}\!-\!\mu)\nabla\Lambda^{1+\alpha}\mathbf{u}_j,\nabla\Lambda^{1+\alpha}\mathbf{u}_j\rangle
\!+\!\mu\|\Lambda^{2+\alpha}\mathbf{u}_j\|^2+\beta_1\kappa\|\nabla\rho_j\|^2\!+\!\beta_2\kappa\sum\limits_{i=1}^3\!\|D_i\nabla\rho_j\|^2\\[2mm]
\!\!\!\!=&\!\!\!\!\beta_1\kappa\|{\rm div}\mathbf{u}\|^2\!+\!\beta_2\kappa\sum\limits_{i=1}^3\|D_i{\rm div}\mathbf{u}\|^2\!-\!\langle\Delta_j(\mathbf{u}\!\cdot\!\nabla\rho),\rho_j\rangle\!-\!\frac{1}{a}\langle\Delta_j(\rho{\rm div}{\mathbf u}),\rho_j\rangle\!-\!\langle\Delta_j(\mathbf{u}\!\cdot\!\nabla \mathbf{u}),\mathbf{u}_j\rangle\\[2mm]
\!\!\!\!&\displaystyle\!\!\!\!-\frac{1}{a}\langle\Delta_j(\rho\nabla\rho),\mathbf{u}_j\rangle-\sum\limits_{i=1}^3\langle D_i\Delta_j(\mathbf{u}\!\cdot\!\nabla\rho),D_i\rho_j\rangle-\sum\limits_{i=1}^3\frac{1}{a}\langle D_i\Delta_j(\rho{\rm div}\mathbf{u}),D_i\rho_j\rangle\\[2mm]
\!\!\!\!&\displaystyle\!\!\!\!-\!\!\sum\limits_{i=1}^3\!(\langle D_i\Delta_j(\mathbf{u}\!\cdot\!\nabla \mathbf{u}),D_i\mathbf{u}_j\rangle\!\!+\!\frac{1}{a}\langle D_i\Delta_j(\rho\nabla\rho),D_i\mathbf{u}_j\rangle)
\!\!-\!\!\sum\limits_{i=1}^3\!\langle\Delta_j\big(D_i(\!\frac{\mu'}{(\kappa\!+\!\frac{1}{a}\rho)^a}\!\!-\!\!\mu\!)\Lambda^{2\alpha}\mathbf{u}\big),D_i\mathbf{u}_j\rangle\\
\end{array}
\end{equation}
\begin{equation*}\label{3.170}
\begin{array}{ll}
&\displaystyle\!\!\!\!-\langle[\Lambda^{1-\alpha},\frac{\mu'}{(\kappa+\frac{1}{a}\rho)^a}-\mu]\nabla\Lambda^{2\alpha}\mathbf{u}_j,\nabla\Lambda^{1+\alpha} \mathbf{u}_j\rangle-\langle[\Delta_j,(\frac{\mu'}{(\kappa+\frac{1}{a}\rho)^a}-\mu)]\Lambda^{2\alpha}\mathbf{u},\mathbf{u}_j\rangle\\[2mm]
&\displaystyle\!\!\!\!-\!\sum\limits_{i=1}^3\!\langle[\Delta_j,\!\frac{\mu'}{(\kappa+\frac{1}{a}\rho)^a}\!-\!\mu]D_i\Lambda^{2\alpha}\mathbf{u},D_i\mathbf{u}_j\rangle\!-\!\langle [\Delta_j,\mathbf{u}\!\cdot\!\nabla]\nabla\rho,\nabla\Lambda^2\rho_j\rangle\!-\!\langle[\Delta_j,\mathbf{u}\!\cdot\!\nabla]\nabla \mathbf{u},\nabla\Lambda^2\mathbf{u}_j\rangle\\[2mm]
&\displaystyle\!\!\!\!-\!\frac{1}{a}\langle[\Delta_j,\rho]\nabla{\rm div}\mathbf{u},\nabla\Lambda^2 \rho_j\rangle\!-\!\frac{1}{a}\langle [\Delta_j,\rho] \nabla^2\rho,\nabla\Lambda^2\mathbf{u}_j\rangle\!-\!\langle[\Delta_j,\frac{\mu'}{(\kappa+\frac{1}{a}\rho)^a}\!
-\!\mu]\nabla\Lambda^{2\alpha}\mathbf{u},\nabla\Lambda^2\mathbf{u}_j\rangle\\[2mm]
:=&\displaystyle\!\!\!\!\sum\limits_{i=1}^{27}I_i.
\end{array}
\end{equation*}
Now, we estimate the terms on the right hand side of \eqref{3.17}. From Lemma \ref{2 2.7}, H\"{o}lder inequality and Young's inequality, we obtain
\begin{equation*}\label{3.180}
\begin{array}{ll}
|I_1|&\!\!\!\!=|-\langle\Delta_j(\mathbf{u}\cdot\nabla\rho),\rho_j\rangle-\frac{1}{a}\langle\Delta_j(\rho{\rm div}\mathbf{u}),\rho_j\rangle|\\[2mm]
&\lesssim2^{-j}\|\Delta_j(\mathbf{u}\cdot\nabla\rho)\|\|\nabla\rho_j\|^2+2^{-j}\|\Delta_j(\rho{\rm div}\mathbf{u})\|\|\nabla\rho_j\|\\[2mm]
&\lesssim2^{-2j}C_{\epsilon_1}\|\Delta_j(\mathbf{u}\cdot\nabla\rho)\|^2+2^{-2j}C_{\epsilon_1}\|\Delta_j(\rho{\rm div}\mathbf{u})\|^2+2\epsilon_1\|\nabla\rho_j\|^2,\\[2mm]
|I_2|&\!\!\!\!=|-\langle\Delta_j(\mathbf{u}\cdot\nabla \mathbf{u}),\mathbf{u}_j\rangle-\frac{1}{a}\langle\Delta_j(\rho\nabla\rho),\mathbf{u}_j\rangle|\\[2mm]
&\lesssim2^{-j\alpha}\|\Delta_j(\mathbf{u}\cdot\nabla \mathbf{u})\|\|\Lambda^\alpha \mathbf{u}_j\|+2^{-j\alpha}\|\Delta_j(\rho\nabla\rho)\|\|\Lambda^\alpha \mathbf{u}_j\|\\[2mm]
&\lesssim2^{-2j\alpha}C_{\epsilon_1}\|\Delta_j(\mathbf{u}\cdot\nabla \mathbf{u})\|^2+2^{-2j\alpha}C_{\epsilon_1}\|\Delta_j(\rho\nabla\rho)\|^2+2\epsilon_1\|\Lambda^\alpha \mathbf{u}_j\|^2.
\end{array}
\end{equation*}
Similary, we have
\begin{equation*}\label{3.181}
\begin{array}{ll}
|I_3|&\!\!\!\!=|-\sum\limits_{i=1}^3\langle D_i\Delta_j(\mathbf{u}\cdot\nabla\rho),D_i\rho_j\rangle-\sum\limits_{i=1}^3\frac{1}{a}\langle D_i\Delta_j(\rho{\rm div}\mathbf{u}),D_i\rho_j\rangle|\\[2mm]
&\lesssim C_{\epsilon_1}\sum\limits_{i=1}^3\|D_i\Delta_j(\mathbf{u}\cdot\nabla\rho)\|^2+C_{\epsilon_1}\sum\limits_{i=1}^3\|D_i\Delta_j(\rho{\rm div}\mathbf{u})\|^2+2\epsilon_1\|\nabla\rho_j\|^2,\\[2mm]
|I_4|&\!\!\!\!=|-\sum\limits_{i=1}^3\langle D_i\Delta_j(\mathbf{u}\cdot\nabla \mathbf{u}),D_i\mathbf{u}_j\rangle-\sum\limits_{i=1}^3\frac{1}{a}\langle D_i\Delta_j(\rho\nabla\rho),D_i\mathbf{u}_j\rangle|\\[2mm]
&\lesssim2^{-2j\alpha}C_{\epsilon_1}\sum\limits_{i=1}^3\big(\|D_i\Delta_j(\mathbf{u}\cdot\nabla \mathbf{u})\|^2+\|D_i\Delta_j(\rho\nabla\rho)\|^2\big)+2\epsilon_1\sum\limits_{i=1}^3\|\Lambda^\alpha D_i\mathbf{u}_j\|^2,\\[2mm]
|I_5|&\!\!\!\!=|-\sum\limits_{i=1}^3\langle\Delta_j\big(D_i(\frac{\mu'}{(\kappa+\frac{1}{a}\rho)^a}-\mu)\Lambda^{2\alpha}\mathbf{u}\big),D_i\mathbf{u}_j\rangle|\\[2mm]
&\lesssim2^{-2j\alpha}C_{\epsilon_1}\sum\limits_{i=1}^3\|\Delta_j\big(D_i(\frac{\mu'}{(\kappa+\frac{1}{a}\rho)^a}-\mu)\Lambda^{2\alpha}\mathbf{u}\big)\|^2+\epsilon_1\sum\limits_{i=1}^3\|\Lambda^\alpha D_i\mathbf{u}_j\|^2,\\
\end{array}
\\[2mm]
|I_{10}|&\!\!\!\!=|\frac{1}{a}\langle\nabla^2\rho_j\nabla\rho,\Lambda^2\mathbf{u}_j\rangle+\frac{1}{a}\langle\nabla\rho \Lambda^2\mathbf{u}_j,\Lambda^2\rho_j\rangle|\\[2mm]
&\lesssim2^{-2j\alpha} C_{\epsilon_1}\|\nabla\rho\|_{L^\infty}^2\|\nabla^2\rho_j\|^2+2^{-2j\alpha}C_{\epsilon_1}\|\nabla\rho\|_{L^\infty}^2\|\Lambda^2\rho_j\|^2+2\epsilon_1\|\Lambda^\alpha\Lambda^2\mathbf{u}_j\|^2,\\[2mm]
\end{array}
\end{equation*}
\begin{equation*}\label{3.190}
\begin{array}{ll}
|I_{11}|&\!\!\!\!=|\langle{\rm div}\Delta_j(\nabla(\frac{\mu'}{(\kappa+\frac{1}{a}\rho)^a}-\mu)\Lambda^{2\alpha} \mathbf{u}),\Lambda^2 \mathbf{u}_j\rangle|\\[2mm]
&\lesssim2^{-2j\alpha}C_{\epsilon_1}\|{\rm div}\Delta_j(\nabla(\frac{\mu'}{(\kappa+\frac{1}{a}\rho)^a}-\mu)\Lambda^{2\alpha} \mathbf{u}\|^2+\epsilon_1\|\Lambda^\alpha\Lambda^2 \mathbf{u}_j\|^2,\\[2mm]
|I_{12}|&\!\!\!\!=|-\beta_1\langle \nabla\Delta_j(\mathbf{u}\cdot\nabla\rho),\mathbf{u}_j\rangle-\beta_1\frac{1}{a}\langle\nabla\Delta_j(\rho{\rm div}\mathbf{u}),\mathbf{u}_j\rangle|\\[2mm]
&\lesssim2^{-2j\alpha}C_{\epsilon_1}\beta_1^2\|\nabla\Delta_j(\mathbf{u}\cdot\nabla\rho)\|^2+2^{-2j\alpha}C_{\epsilon_1}\beta_1^2\|\nabla\Delta_j(\rho{\rm div}\mathbf{u})\|^2+2\epsilon_1\|\Lambda^\alpha \mathbf{u}_j\|^2,\\[2mm]
|I_{13}|&\!\!\!\!=|-\beta_1\frac{1}{a}\langle\Delta_j(\rho\nabla\rho),\nabla\rho_j\rangle-\beta_1\langle\Delta_j(\mathbf{u}\!\cdot\!\nabla\mathbf{u}),\nabla\rho_j\rangle
-\beta_1\langle\Delta_j\big((\frac{\mu'}{(\kappa+\frac{1}{a}\rho)^a}\!-\!\mu)\Lambda^{2\alpha}\mathbf{u}\big),\nabla\rho_j\rangle|\\[2mm]
&\lesssim C_{\epsilon_1}\beta_1^2\big(\|\Delta_j(\rho\nabla\rho)\|^2\!+\!\|\Delta_j(\mathbf{u}\!\cdot\!\nabla \mathbf{u})\|^2\!+\!\|\Delta_j\big((\frac{\mu'}{(\kappa+\frac{1}{a}\rho)^a}\!-\!\mu)\Lambda^{2\alpha}\mathbf{u}\big)\|^2\!\big)+\!3\epsilon_1\|\nabla\rho_j\|^2,\\[2mm]
\end{array}
\end{equation*}
\begin{equation*}\label{3.19}
\begin{array}{ll}
|I_{14}|&\!\!\!\!=|-\beta_1\mu\langle\nabla\rho_j,\Lambda^{2\alpha}\mathbf{u}_j\rangle|\lesssim C_{\epsilon_1}\beta_1^2\|\Lambda^{2\alpha}\mathbf{u}_j\|^2+\epsilon_1\|\nabla\rho_j\|^2,\\[2mm]
|I_{15}|&\!\!\!\!=|\beta_2\sum\limits_{i=1}^3\langle D_i\Delta_j(\mathbf{u}\!\cdot\!\nabla\rho),{\rm div}D_i\mathbf{u}_j\rangle+\beta_2\frac{1}{a}\sum\limits_{i=1}^3\langle D_i\Delta_j(\rho{\rm div}\mathbf{u}),{\rm div}D_i\mathbf{u}_j\rangle|\\[2mm]
&\!\!\!\!\lesssim 2^{-2j\alpha}C_{\epsilon_1}\beta_2^2\big(\sum\limits_{i=1}^3\|D_i\Delta_j(\mathbf{u}\!\cdot\!\nabla\rho)\|^2+\sum\limits_{i=1}^3\|D_i\Delta_j(\rho{\rm div}\mathbf{u})\|^2\big)+2\epsilon_1\sum\limits_{i=1}^3\|\Lambda^\alpha{\rm div}D_i\mathbf{u}_j\|^2,\\[2mm]
|I_{16}|&\!\!\!\!=|-\beta_2\sum\limits_{i=1}^3\frac{1}{a}\langle D_i\Delta_j(\rho\nabla\rho),D_i\nabla\rho_j\rangle-\beta_2\sum\limits_{i=1}^3\langle\Delta_jD_i(\mathbf{u}\!\cdot\!\nabla \mathbf{u}),D_i\nabla\rho_j\rangle,\\[2mm]
&\!\!\!\!\quad-\beta_2\sum\limits_{i=1}^3\langle\Delta_jD_i\big((\frac{\mu'}{(\kappa+\frac{1}{a}\rho)^a}\!-\!\mu)\Lambda^{2\alpha}\mathbf{u}\big),D_i\nabla\rho_j\rangle|\\[2mm]
&\!\!\!\!\lesssim C_{\epsilon_1}\beta_2^2\sum\limits_{i=1}^3\|D_i\Delta_j(\rho\nabla\rho)\|^2+C_{\epsilon_1}\beta_2^2\sum\limits_{i=1}^3\|D_i\Delta_j(\mathbf{u}\!\cdot\!\nabla \mathbf{u})\|^2\\[2.5mm]
&\!\!\!\!\quad+C_{\epsilon_1}\beta_2^2\|\sum\limits_{i=1}^3D_i\Delta_j\big((\frac{\mu'}{(\kappa+\frac{1}{a}\rho)^a}\!-\!\mu)\Lambda^{2\alpha}\mathbf{u}\big)\|^2
+3\epsilon_1\sum\limits_{i=1}^3\|D_i\nabla\rho_j\|^2,\\
\end{array}
\end{equation*}
By using H\"{o}lder inequality, Young's inequality and Commutator Estimates(\cite{C,CTT}, Lemma 2.1 in \cite{Wangs2}), we get
\begin{equation*}\label{3.20}
\begin{array}{ll}
|I_{18}|&\!\!\!\!=|\!-\!\langle[\Lambda^\alpha,\frac{\mu'}{(\kappa+\frac{1}{a}\rho)^a}-\mu)]\mathbf{u}_j,\Lambda^\alpha u_j\rangle|\lesssim\|\Lambda^\alpha(\frac{\mu'}{(\kappa+\frac{1}{a}\rho)^a}\!-\!\mu)\|_{L^\infty}\|\mathbf{u}_j\|\|\Lambda^\alpha \mathbf{u}_j\|\\[2mm]
&\!\!\!\!\lesssim2^{-2j\alpha}C_{\epsilon_1}\|\Lambda^\alpha\rho\|_{L^\infty}^2\|\Lambda^\alpha \mathbf{u}_j\|^2+\epsilon_1\|\Lambda^\alpha \mathbf{u}_j\|^2.
\end{array}
\end{equation*}
Then, $I_{19}-I_{27}$ can be bounded in the same way as $I_{18}$:
\begin{equation*}\label{3.20}
\begin{array}{ll}
|I_{19}|&\!\!\!\!=|\!-\!\!\!\sum\limits_{i=1}^3\langle[\Lambda^\alpha,\frac{\mu'}{(\kappa+\frac{1}{a}\rho)^a}\!-\!\mu]D_i\mathbf{u}_j,D_i\Lambda^\alpha \mathbf{u}_j\rangle|
\lesssim C_{\epsilon_1}\|\Lambda^\alpha\rho\|_{L^\infty}^2\|\nabla \mathbf{u}_j\|^2+\epsilon_1\sum\limits_{i=1}^3\|D_i\Lambda^\alpha \mathbf{u}_j\|^2,\\[2mm]
|I_{20}|&\!\!\!\!=|\!-\!\!\langle[\Lambda\!^{1\!-\alpha}\!,\!\frac{\mu'}{(\kappa+\frac{1}{a}\rho)^a}\!-\!\mu]\nabla\!\Lambda\!^{2\alpha}\mathbf{u}_j,\!\!\nabla\!\Lambda\!^{1\!+\alpha} \mathbf{u}_j\rangle|
\lesssim C_{\epsilon_1}\|\Lambda^{1-\alpha}\rho\|_{L^\infty}^2\|\nabla\Lambda^{2\alpha}\mathbf{u}_j\|^2+\epsilon_1\|\nabla\Lambda^{1+\alpha} \mathbf{u}_j\|^2,\\[2mm]
|I_{21}|&\!\!\!\!=|-\langle[\Delta_j,(\frac{\mu'}{(\kappa+\frac{1}{a}\rho)^a}-\mu)]\Lambda^{2\alpha}\mathbf{u},\mathbf{u}_j\rangle|
\lesssim2^{-js}\|\nabla\rho\|_{H^s}\|\nabla\Lambda^{2\alpha-2}\mathbf{u}\|_{H^s}\|\Lambda^\alpha \mathbf{u}_j\|,\\[2mm]
|I_{22}|&\!\!\!\!=|\!-\!\sum\limits_{i=1}^3\langle[\Delta_j,\frac{\mu'}{(\kappa+\frac{1}{a}\rho)^a}\!-\!\mu]D_i\Lambda^{2\alpha}\mathbf{u},D_i\mathbf{u}_j\rangle|
\lesssim\sum\limits_{i=1}^32^{-js}\|\nabla\rho\|_{H^s}\|D_i\nabla\Lambda^{2\alpha-2}\mathbf{u}\|_{H^s}\|D_i\mathbf{u}_j\|,\\[2mm]
|I_{23}|&\displaystyle\!\!\!\!=|-\langle [\Delta_j,\mathbf{u}\cdot\nabla]\nabla\rho,\nabla\Lambda^2\rho_j\rangle|
\lesssim2^{-js}\|\nabla \mathbf{u}\|_{H^s}\|\nabla\rho\|_{H^s}\|\nabla\Lambda^2\rho_j\|,\\[2mm]
|I_{24}|&\!\!\!\!=|\!\!-\!\!\langle [\Delta_j,\mathbf{u}\cdot\nabla]\nabla\mathbf{u},\nabla\!\Lambda^2\mathbf{u}_j\rangle|
\lesssim\!2^{-js}\|\nabla \mathbf{u}\|_{H^s}\|\nabla \mathbf{u}\|_{H^s}\|\nabla\!\Lambda^2\mathbf{u}_j\|,\\
\end{array}
\end{equation*}
Inserting the above inequalities about $I_1-I_{27}$ into (\ref{3.17}), we have
\begin{equation}\label{3.22}
\begin{array}{ll}
\!\!&\displaystyle\!\!\!\!\frac{1}{2}\frac{d}{dt}(\|\rho_j\|^2+\|\mathbf{u}_j\|^2+\sum\limits_{i=1}^3\|D_i\rho_j\|^2+\sum\limits_{i=1}^3\|D_i\mathbf{u}_j\|^2+\|\Lambda^2\rho_j\|^2+\|\Lambda ^2\mathbf{u}_j\|^2+2\beta_1\langle\nabla\rho_j,\mathbf{u}_j\rangle\\[2mm]
\!\!&\displaystyle\!\!\!\!+2\beta_2\sum\limits_{i=1}^3\langle D_i\nabla\rho_j,D_i\mathbf{u}_j\rangle)+\langle\frac{\mu'}{(\kappa+\frac{1}{a}\rho)^a}\Lambda^\alpha \mathbf{u}_j,\Lambda^\alpha \mathbf{u}_j\rangle+\sum\limits_{i=1}^3\langle\frac{\mu'}{(\kappa+\frac{1}{a}\rho)^a}D_i\Lambda^\alpha \mathbf{u}_j,D_i\Lambda^\alpha \mathbf{u}_j\rangle\\[2mm]
\!\!&\displaystyle\!\!\!\!+\langle(\frac{\mu'}{(\kappa+\frac{1}{a}\rho)^a}\!-\!\mu)\nabla\Lambda^{1+\alpha}\mathbf{u}_j,\nabla\Lambda^{1+\alpha}\mathbf{u}_j\rangle
\!+\!\mu\|\Lambda^{2+\alpha}\mathbf{u}_j\|^2\!+\!\beta_1\kappa\|\nabla\rho_j\|^2\!+\!\beta_2\kappa\sum\limits_{i=1}^3\!\|D_i\nabla\rho_j\|^2\\[2mm]
\lesssim
\!\!&\displaystyle\!\!\beta_1\kappa\|{\rm div}\mathbf{u}\|^2\!\!+\!\beta_2\kappa\!\sum\limits_{i=1}^3\!\|D_i{\rm div}\mathbf{u}\|^2\!\!+\!8\epsilon_1\|\nabla\!\rho_j\|^2\!\!+\!5\epsilon_1\|\Lambda\!^2\rho_j\|^2\!\!+\!4\epsilon_1\!\!\sum\limits_{i=1}^3\!\|D_i\nabla\!\rho_j\|^2\!\!+\!5\epsilon_1\|\Lambda\!^\alpha \mathbf{u}_j\|^2\\[2mm]
\!\!&\displaystyle\!\!\!\!+4\epsilon_1\sum\limits_{i=1}^3\|\Lambda^\alpha D_i\mathbf{u}_j\|^2\!+\!7\epsilon_1\|\Lambda^\alpha\Lambda^2\mathbf{u}_j\|^2\!+\!2\epsilon_1\sum\limits_{i=1}^3\!\|\Lambda^\alpha{\rm div}D_i\mathbf{u}_j\|^2\!+\!\epsilon_1\|\nabla\Lambda^{1+\alpha} \mathbf{u}_j\|^2\\[2mm]
\!\!&\displaystyle\!\!\!\!+C_{\epsilon_1}\beta_1^2\|\Lambda^{2\alpha}\mathbf{u}_j\|^2\!+\!C_{\epsilon_1}\beta_2^2\sum\limits_{i=1}^3\|D_i\Lambda^{2\alpha}\mathbf{u}_j\|^2
\!+\!2^{-2j}C_{\epsilon_1}(\|\Delta_j(\mathbf{u}\cdot\nabla\rho)\|^2+\|\Delta_j(\rho{\rm div}\mathbf{u})\|^2)\\[2mm]
\!\!&\displaystyle\!\!\!\!+2^{-2j\alpha}C_{\epsilon_1}(\|\Delta_j(\mathbf{u}\!\cdot\!\nabla \mathbf{u})\|^2\!+\!\|\Delta_j(\rho\nabla\rho)\|^2)\!+\!C_{\epsilon_1}\!\sum\limits_{i=1}^3\!\Big(\|D_i\Delta_j(\mathbf{u}\!\cdot\!\nabla\rho)\|^2\!+\!\|D_i\Delta_j(\rho{\rm div}\mathbf{u})\|^2\Big)\\[2mm]
\!\!&\displaystyle\!\!\!\!+2^{-\!2j\alpha}C\!_{\epsilon_1}\!\sum\limits_{i=1}^3\Big(\|D_i\Delta_j(\!\mathbf{u}\!\cdot\!\nabla\mathbf{u}\!)\|^2\!+\!\|D_i\Delta_j\big(\rho\nabla\rho)\|^2
\!+
\big\|\Delta_j(\!D_i(\!\frac{\mu'}{(\!\kappa\!+\!\frac{1}{a}\rho\!)^a}\!-\!\mu\!)
\Lambda\!^{2\alpha}\!\mathbf{u}\!)\big\|^2\Big)\\[2mm]
\!\!&\displaystyle\!\!\!\!+C_{\epsilon_1}(\|{\rm div}\Delta_j(\nabla \mathbf{u}\!\cdot\!\nabla\rho)\|^2\!+\!\|{\rm div}\Delta_j(\nabla\rho{\rm div}\mathbf{u})\|^2)
+2^{-2j\alpha}C_{\epsilon_1}\!\|\nabla \mathbf{u}\|_{L^\infty}^2\|\nabla^2\mathbf{u}_j\|^2\\[2mm]
\!\!&\displaystyle\!\!\!\!+2^{-2j\alpha}C_{\epsilon_1}\!\|{\rm div}u\|_{L\!^\infty}^2\|\Lambda^2\mathbf{u}_j\|^2\!+\!2^{-2j\alpha}C_{\epsilon_1}\|{\rm div}\Delta_j(\nabla\mathbf{u}\!\cdot\!\nabla \mathbf{u})\|^2\!+\!2^{-2j\alpha}C_{\epsilon_1}\|{\rm div}\Delta_j(\nabla\rho\nabla\rho)\|^2\\[2mm]
\!\!&\displaystyle\!\!\!\!+C_{\epsilon_1}\|\nabla \mathbf{u}\|_{L^\infty}^2\|\nabla^2\rho_j\|^2+C_{\epsilon_1}\|{\rm div}\mathbf{u}\|_{L^\infty}^2\|\Lambda^2\rho_j\|^2
+C_{\epsilon_1}\|\nabla\rho\|_{L^\infty}^2\|\nabla{\rm div}u_j\|^2\\[2mm]
\!\!&\displaystyle\!\!\!\!+2^{-2j\alpha} C_{\epsilon_1}\|\nabla\rho\|_{L^\infty}^2(\|\nabla^2\rho_j\|^2+\|\Lambda^2\rho_j\|^2)+2^{-2j\alpha}C_{\epsilon_1}\|{\rm div}\Delta_j(\nabla(\frac{\mu'}{(\kappa+\frac{1}{a}\rho)^a}-\mu)\Lambda^{2\alpha} \mathbf{u})\|^2\\[2mm]
\!\!&\displaystyle\!\!\!\!+2^{-2j\alpha}C_{\epsilon_1}\beta_1^2(\|\nabla\Delta_j(\mathbf{u}\cdot\nabla\rho)\|^2+\|\nabla\Delta_j(\rho{\rm div}\mathbf{u})\|^2)\\[2mm]
\!\!&\displaystyle\!\!\!\!+C_{\epsilon_1}\beta_1^2\Big(\|\Delta_j(\rho\nabla\rho)\|^2+\|\Delta_j(\mathbf{u}\cdot\nabla \mathbf{u})\|^2+\|\Delta_j\big((\frac{\mu'}{(\kappa+\frac{1}{a}\rho)^a}-\mu)\Lambda^{2\alpha}\mathbf{u}\big)\|^2\Big)\\
\end{array}
\end{equation}
By using $\|\Lambda^\alpha \mathbf{u}_j\|^2\lesssim\|\mathbf{u}_j\|^2+\|\nabla \mathbf{u}_j\|^2$, we get
\begin{equation}\label{3.23}
\begin{array}{ll}
&\|\Lambda^{2\alpha} \mathbf{u}_j\|^2\lesssim\|\Lambda^\alpha \mathbf{u}_j\|^2+\|\nabla\Lambda^\alpha \mathbf{u}_j\|^2,\\[2mm]
&\|D_i\Lambda^{2\alpha} \mathbf{u}_j\|^2\lesssim\|D_i\Lambda^\alpha \mathbf{u}_j\|^2+\|D_i\nabla\Lambda^\alpha \mathbf{u}_j\|^2.
\end{array}
\end{equation}
Choosing $\beta_1,\beta_2$ and $\epsilon_1$ small enough, and combining (\ref{3.23}) with the following facts:
\begin{equation}\label{3.24}
\begin{array}{ll}
&\displaystyle\!\!\!\!\|(\rho_j,{\bf u})\|^2+\!\sum\limits_{i=1}^3\|D_i(\rho_j,\mathbf{u}_j)\|^2\!+\!\|\Lambda\!^2(\rho_j,\mathbf{u}_j)\|^2
+2\beta_1\langle\nabla\!\rho_j,\mathbf{u}_j\rangle+2\beta_2\sum\limits_{i=1}^3\langle D_i\!\nabla\!\rho_j,D_i\mathbf{u}_j\rangle\\[2mm]
&\displaystyle\!\!\!\!\approx\|\rho_j\|_{H^2}+\|\mathbf{u}_j\|_{H^2},\\[2mm]
&\displaystyle\!\!\!\!\big\langle\frac{\mu'}{(\kappa+\frac{1}{a}\rho)^a}\Lambda^\alpha \mathbf{u}_j,\Lambda\!^\alpha \mathbf{u}_j\big\rangle+\sum\limits_{i=1}^3\big\langle\frac{\mu'}{(\kappa+\frac{1}{a}\rho)^a}D_i\Lambda\!^\alpha\mathbf{u}_j,D_i\Lambda\!^\alpha \mathbf{u}_j\big\rangle\\[2mm]
&\displaystyle\!\!\!\!+\big\langle(\!\frac{\mu'}{(\kappa+\frac{1}{a}\rho)^a}-\mu)\nabla\!\Lambda^{1+\alpha}\mathbf{u}_j,\nabla\Lambda^{1+\alpha}\mathbf{u}_j\big\rangle
+\mu\|\Lambda^{2+\alpha}\mathbf{u}_j\|^2\approx\|\Lambda^\alpha \mathbf{u}_j\|_{H^2}^2,\\[2mm]
&\displaystyle\!\!\!\!\|\nabla\rho_j\|^2+\sum\limits_{i=1}^3\|D_i\nabla\rho_j\|^2\approx\|\nabla\rho_j\|_{H^1}^2,\\[2mm]
\end{array}
\end{equation}
for any $j\geq0$, we can infer from (\ref{3.22}) that
\begin{equation}\label{3.25}
\begin{array}{ll}
\!\!&\displaystyle\|\rho_j\|_{H^2}+\|\mathbf{u}_j\|_{H^2}+\int_0^t\|\nabla\rho_j\|_{H^1}^2+\|\Lambda^\alpha \mathbf{u}_j\|_{H^2}^2d\tau\\[2mm]
\lesssim\!\!\!\!&\displaystyle\|\Delta_j\rho_0\|_{H^2}+\|\Delta_j\mathbf{u}_0\|_{H^2}+\int_0^t\bigg\{\|\Delta_j(\mathbf{u}\cdot\nabla\rho)\|^2+\|\Delta_j(\rho{\rm div}\mathbf{u})\|^2+\|\Delta_j(\mathbf{u}\cdot\nabla \mathbf{u})\|^2\\[2mm]
\!\!&\displaystyle\!+\|\Delta_j(\rho\nabla\rho)\|^2+\|\nabla\Delta_j(\mathbf{u}\cdot\nabla\rho)\|^2+\|\nabla\Delta_j(\rho{\rm div}\mathbf{u})\|^2+\|\nabla\Delta_j(\mathbf{u}\cdot\nabla \mathbf{u})\|^2\\[2mm]
\!\!&\displaystyle\!+\|\nabla\Delta_j(\rho\nabla\rho)\|^2+\sum\limits_{i=1}^3\|\Delta_j\big(D_i(\frac{\mu'}{(\kappa+\frac{1}{a}\rho)^a}-\mu)\Lambda^{2\alpha}\mathbf{u}\big)\|^2+\|{\rm div}\Delta_j(\nabla \mathbf{u}\cdot\nabla\rho)\|^2\\[2mm]
\!\!&\displaystyle\!+\|{\rm div}\Delta_j(\nabla\rho{\rm div}\mathbf{u})\|^2+\|\nabla \mathbf{u}\|_{L^\infty}^2\|\nabla^2 \mathbf{u}_j\|^2+\|{\rm div}\mathbf{u}\|_{L^\infty}^2\|\Lambda^2\mathbf{u}_j\|^2+\|{\rm div}\Delta_j(\nabla \mathbf{u}\cdot\nabla \mathbf{u})\|^2\\[2mm]
\!\!&\displaystyle\!+\|{\rm div}\Delta_j(\nabla\rho\nabla\rho)\|^2+\|\nabla \mathbf{u}\|_{L^\infty}^2\|\nabla^2\rho_j\|^2+\|{\rm div}\mathbf{u}\|_{L^\infty}^2\|\Lambda^2\rho_j\|^2
+\|\nabla\rho\|_{L^\infty}^2\|\nabla{\rm div}\mathbf{u}_j\|^2\\[2mm]
\!\!&\displaystyle\!+\|\nabla\rho\|_{L^\infty}^2\|\nabla^2\rho_j\|^2+\|\nabla\rho\|_{L^\infty}^2\|\Lambda^2\rho_j\|^2+\|{\rm div}\Delta_j(\nabla(\frac{\mu'}{(\kappa+\frac{1}{a}\rho)^a}-\mu)\Lambda^{2\alpha} \mathbf{u}\|^2\\[2mm]
\!\!&\displaystyle\!+\|\Delta_j\big((\frac{\mu'}{(\kappa\!+\!\frac{1}{a}\rho)^a}\!-\!\mu)\Lambda^{2\alpha}\mathbf{u}\big)\|^2
\!+\!\sum\limits_{i=1}^3\|D_i\Delta_j\big((\frac{\mu'}{(\kappa\!+\!\frac{1}{a}\rho)^a}\!-\!\mu)\Lambda^{2\alpha}\mathbf{u}\big)\|^2
\!+\!\|\Lambda^\alpha\rho\|_{L^\infty}^2\|\Lambda^\alpha \mathbf{u}_j\|^2\\[2mm]
\!\!&\displaystyle\!+\|\Lambda^\alpha\rho\|_{L^\infty}^2\|\nabla \mathbf{u}_j\|^2+\|\Lambda^{1-\alpha}\rho\|_{L^\infty}^2\|\nabla\Lambda^{2\alpha}\mathbf{u}_j\|^2
+2^{-js}\|\nabla\rho\|_{H^s}\|\nabla\Lambda^{2\alpha-2}\mathbf{u}\|_{H^s}\|\Lambda^\alpha \mathbf{u}_j\|\\[2mm]
\!\!&\displaystyle\!+\sum\limits_{i=1}^32^{-js}\|\nabla\rho\|_{H^s}\|D_i\nabla\Lambda^{2\alpha-2}\mathbf{u}\|_{H^s}\|D_i\mathbf{u}_j\|+2^{-js}\|\nabla \mathbf{u}\|_{H^s}\|\nabla\rho\|_{H^s}\|\nabla\Lambda^2\rho_j\|\\[2mm]
\!\!&\displaystyle\!+2^{-js}\|\nabla \mathbf{u}\|_{H^s}\|\nabla \mathbf{u}\|_{H^s}\|\nabla\Lambda^2\mathbf{u}_j\|+2^{-js}\|\nabla\rho\|_{H^s}\|{\rm div}\mathbf{u}\|_{H^s}\|\nabla\Lambda^2\rho_j\|\\[2mm]
\!\!&\displaystyle\!+2^{-js}\|\nabla\rho\|_{H^s}\|\nabla\rho\|_{H^s}\|\nabla\Lambda^2\mathbf{u}_j\|+2^{-js}\|\nabla\rho\|_{H^s}\|\nabla\rho\|_{H^s}\|\nabla\Lambda^2\mathbf{u}_j\|\bigg\}d\tau.
\end{array}
\end{equation}
By the same argument as in (\ref{3.17}), we have
\begin{equation}\label{3.26}
\begin{array}{ll}
&\!\!\!\!\displaystyle\frac{1}{2}\frac{d}{dt}(\|\rho\|^2\!+\!\|\mathbf{u}\|^2\!+\!\sum\limits_{i=1}^3\|D_i\rho\|^2\!+\!\sum\limits_{i=1}^3\|D_i\mathbf{u}\|^2
\!+\!2\beta_1\langle\nabla\rho,\mathbf{u}\rangle)\!+\!\langle\frac{\mu'}{(\!\kappa\!+\!\frac{1}{a}\rho\!)^a}\Lambda^\alpha \mathbf{u},\Lambda^\alpha \mathbf{u}\rangle\\[2mm]
&\displaystyle\!\!\!\!+\sum\limits_{i=1}^3\langle\frac{\mu'}{(\kappa+\frac{1}{a}\rho)^a}D_i\Lambda^\alpha \mathbf{u},D_i\Lambda^\alpha \mathbf{u}\rangle+\beta_1\kappa\|\nabla\rho\|^2\\[2mm]
=&\displaystyle\!\!\!\!-\langle \mathbf{u}\!\cdot\!\nabla\rho,\rho\rangle\!+\!\frac{1}{a}\langle\rho\nabla\!\rho,\mathbf{u}\rangle\!-\!\langle \mathbf{u}\!\cdot\!\nabla \mathbf{u},\mathbf{u}\rangle\!-\!\langle([\Lambda\!^\alpha,\frac{\mu'}{(\!\kappa\!+\!\frac{1}{a}\rho\!)^a}\!-\!\mu]\mathbf{u},\Lambda\!^\alpha \mathbf{u}\rangle
\!-\!\sum\limits_{i=1}^3\langle D_i(\mathbf{u}\!\cdot\!\nabla\rho),D_i\rho\rangle\\[2mm]
&\displaystyle\!\!\!\!-\!\sum\limits_{i=1}^3\frac{1}{a}\langle D_i\rho{\rm div}\mathbf{u},D_i\rho\rangle-\sum\limits_{i=1}^3\langle D_i(\mathbf{u}\!\cdot\!\nabla \mathbf{u}),D_i\mathbf{u}\rangle\!-\!\sum\limits_{i=1}^3\langle D_i(\frac{\mu'}{(\kappa+\frac{1}{a}\rho)^a}\!-\!\mu)\Lambda^{2\alpha}\mathbf{u},D_i\mathbf{u}\rangle
\end{array}
\end{equation}
\begin{equation}
\begin{array}{ll}
&\displaystyle\!\!\!\!-\!\!\sum\limits_{i=1}^3\langle[\Lambda\!^\alpha,\!\frac{\mu'}{(\!\kappa\!+\!\frac{1}{a}\rho\!)^a}\!-\!\mu]D_i\mathbf{u},D_i\Lambda^\alpha \mathbf{u}\rangle
\!+\!\beta_1\kappa\|{\rm div}\mathbf{u}\|^2\!\!-\!\beta_1\!\langle \nabla(\mathbf{u}\!\cdot\!\nabla\rho),\mathbf{u}\rangle\!\!-\!\beta_1\frac{1}{a}\langle\nabla(\rho{\rm div}\mathbf{u}),\mathbf{u}\rangle\\[2mm]
&\displaystyle\!\!\!\!-\beta_1\frac{1}{a}\langle\rho\nabla\!\rho,\nabla\rho\rangle-\beta_1\langle \mathbf{u}\!\cdot\!\nabla \mathbf{u},\nabla\rho\rangle-\beta_1\langle(\frac{\mu'}{(\kappa+\frac{1}{a}\rho)^a}-\mu)\Lambda^{2\alpha}\mathbf{u},\nabla\rho\rangle
-\beta_1\mu\langle\nabla\rho,\Lambda^{2\alpha}\mathbf{u}\rangle\\[2mm]
:=&\!\!\!\!\displaystyle\sum\limits_{i=1}^{11}S_i.
\end{array}
\end{equation}
From Commutator Estimates(\cite{C,CTT}), Lemma \ref{2 2.3} and Lemma \ref{2 2.6}, we bound the terms $S_1-S_3$ as follows:
\begin{equation*}\label{3.27}
\begin{array}{ll}
|S_1|&\!\!\!\!=|-\langle \mathbf{u}\cdot\nabla\rho,\rho\rangle+\frac{1}{a}\langle\rho\nabla\rho,\mathbf{u}\rangle|\lesssim\|\mathbf{u}\|_{L^3}\|\nabla\rho\|_{L^2}\|\rho\|_{L^6}\lesssim\|\nabla \mathbf{u}\|^2\|\rho\|_{H^1}^2+\epsilon_1\|\nabla\rho\|^2,\\[2mm]
|S_2|&\!\!\!\!=|-\langle \mathbf{u}\cdot\nabla \mathbf{u},\mathbf{u}\rangle|
\lesssim\|\mathbf{u}\|_{L^\frac{3}{\alpha}}\|\nabla \mathbf{u}\|_{L^2}\|\mathbf{u}\|_{L^\frac{6}{3-2\alpha}}
\lesssim\|\nabla \mathbf{u}\|^2\|\mathbf{u}\|_{H^1}^2+\epsilon_1\|\Lambda^\alpha \mathbf{u}\|^2,\\[2mm]
|S_3|&\!\!\!\!=|-\langle([\Lambda^\alpha,\frac{\mu'}{(\kappa+\frac{1}{a}\rho)^a}-\mu)]\mathbf{u},\Lambda^\alpha \mathbf{u}\rangle|
\lesssim\|\Lambda^\alpha(\frac{\mu'}{(\kappa+\frac{1}{a}\rho)^a}-\mu)\|_{L^\frac{3}{\alpha}}\|\Lambda^\alpha \mathbf{u}\|_{L^2}\|\mathbf{u}\|_{L^\frac{6}{3-2\alpha}}\\[2mm]
&\!\!\!\!\lesssim\|\Lambda^\alpha\rho\|_{H^1}^2\|\Lambda^\alpha \mathbf{u}\|^2+\epsilon_1\|\Lambda^\alpha \mathbf{u}\|^2,
\end{array}
\end{equation*}
where we also used H\"{o}lder inequality and Young's inequality.
Similary, we get
\begin{equation}\label{3.28}
\begin{array}{ll}
|S_4|&\!\!\!\!=|-\sum\limits_{i=1}^3\langle D_i(\mathbf{u}\cdot\nabla\rho),D_i\rho\rangle|
\lesssim\|\Lambda^2\rho\|^2\|\mathbf{u}\|_{H^s}^2+\epsilon_1\|\nabla\rho\|^2,\\[2.5mm]
|S_5|&\!\!\!\!=|-\sum\limits_{i=1}^3\frac{1}{a}\langle D_i\rho{\rm div}\mathbf{u},D_i\rho\rangle|
\lesssim\|\nabla\rho\|^2\|{\rm div}\mathbf{u}\|_{H^s}^2+\epsilon_1\|\nabla\rho\|^2,\\[2.5mm]
|S_6|&\!\!\!\!=|\!-\!\sum\limits_{i=1}^3\langle D_i(\mathbf{u}\!\cdot\!\nabla \mathbf{u}),D_i\mathbf{u}\rangle|
\lesssim\|\nabla \mathbf{u}\|_{H^1}^2\|\Lambda^2\mathbf{u}\|^2+\epsilon_1\|\Lambda^\alpha \mathbf{u}\|^2,\\[2mm]
|S_7|&\!\!\!\!=\Big|\!-\!\sum\limits_{i=1}^3\langle D_i(\frac{\mu'}{(\kappa+\!\frac{1}{a}\rho)^a}\!-\!\mu)\Lambda^{2\alpha}\mathbf{u},D_i\mathbf{u}\rangle\Big|
\lesssim\sum\limits_{i=1}^3\|D_i\rho\|_{H^1}^2\|\Lambda^{2\alpha}\mathbf{u}\|^2+\epsilon_1\sum\limits_{i=1}^3\|D_i\Lambda^\alpha \mathbf{u}\|^2,\\
\end{array}
|S_9|&\!\!\!\!=|-\beta_1\langle \nabla(\mathbf{u}\!\cdot\!\nabla\rho),\mathbf{u}\rangle\!-\!\beta_1\frac{1}{a}\langle\nabla(\rho{\rm div}\mathbf{u}),\mathbf{u}\rangle|
\\[2.6mm]
&\!\!\!\!\lesssim\beta_1^2\|\nabla \mathbf{u}\|^2\|{\rm div}\mathbf{u}\|_{H^1}^2+\beta_1^2\|{\rm div}\mathbf{u}\|_{H^1}^2\|{\rm div}\mathbf{u}\|^2+2\epsilon_1\|\nabla\rho\|^2,\\[2.62mm]
|S_{10}|&\!\!\!\!=|-\beta_1\frac{1}{a}\langle\rho\nabla\rho,\nabla\rho\rangle-\beta_1\langle \mathbf{u}\cdot\nabla \mathbf{u},\nabla\rho\rangle|
\\[2.6mm]
&\!\!\!\!\lesssim\beta_1^2\|\rho\|_{H^s}^2\|\nabla\rho\|^2+\beta_1^2\|\mathbf{u}\|_{H^s}^2\|\nabla \mathbf{u}\|^2+2\epsilon_1\|\nabla\rho\|^2,\\[2mm]
|S_{11}|&\!\!\!\!=|-\beta_1\langle(\frac{\mu'}{(\kappa+\frac{1}{a}\rho)^a}\!-\mu)\Lambda^{2\alpha}\mathbf{u},\nabla\rho\rangle
-\beta_1\mu\langle\nabla\rho,\Lambda^{2\alpha}\mathbf{u}\rangle|\\[2.6mm]
&\!\!\!\!\lesssim\beta_1^2\|\rho\|_{L^\infty}^2\|\Lambda^{2\alpha}\mathbf{u}\|^2+\beta_1^2\mu^2\|\Lambda^{2\alpha}\mathbf{u}\|^2+\epsilon_1\|\nabla\rho\|^2.
\end{array}
\end{equation}
A similar computation as in (\ref{3.25}) gives
\begin{equation}\label{3.29}
\begin{array}{ll}
&\displaystyle\!\!\!\!\|\rho\|_{H^1}^2+\|\mathbf{u}\|_{H^1}^2+\int_0^t(\|\nabla\rho\|^2+\|\Lambda^\alpha \mathbf{u}\|_{H^1}^2)d\tau\\[2.6mm]
\lesssim&\!\!\!\!\displaystyle\|\rho_0\|_{H^1}^2+\|\mathbf{u}_0\|_{H^1}^2\int_0^t\{\|\nabla \mathbf{u}\|^2\|\rho\|_{H^1}^2+\|\nabla \mathbf{u}\|^2\|\mathbf{u}\|_{H^1}^2+\|\Lambda^\alpha\rho\|_{H^1}^2\|\Lambda^\alpha \mathbf{u}\|^2\\[2.6mm]
&\!\!\!\!\displaystyle+\|\Lambda^2\rho\|^2\|\mathbf{u}\|_{H^s}^2+\|\nabla\rho\|^2\|{\rm div}\mathbf{u}\|_{H^s}^2+\|\nabla \mathbf{u}\|_{H^1}^2\|\Lambda^2\mathbf{u}\|^2+\sum\limits_{i=1}^3\|D_i\rho\|_{H^1}^2\|\Lambda^{2\alpha}\mathbf{u}\|^2\\[2.6mm]
&\!\!\!\!\displaystyle+\sum\limits_{i=1}^3\|\Lambda^\alpha\rho\|_{H^1}^2\|\Lambda^\alpha D_i\mathbf{u}\|^2+\beta_1^2\|\nabla \mathbf{u}\|^2\|{\rm div}\mathbf{u}\|_{H^1}^2+\beta_1^2\|{\rm div}\mathbf{u}\|_{H^1}^2\|{\rm div}\mathbf{u}\|^2\\[3mm]
&\displaystyle\!\!\!\!+\beta_1^2\|\rho\|_{H^s}^2\|\nabla\rho\|^2+\beta_1^2\|\mathbf{u}\|_{H^s}^2\|\nabla \mathbf{u}\|^2+\beta_1^2\|\rho\|_{H^s}^2\|\Lambda^{2\alpha}\mathbf{u}\|^2\}d\tau\\[2.6mm]
\lesssim&\!\!\!\!\displaystyle\|\rho_0\|_{H^1}^2+\|\mathbf{u}_0\|_{H^1}^2+\int_0^t(\|\rho\|_{H^{s+1}}^2+\|\mathbf{u}\|_{H^{s+1}}^2)(\|\nabla\rho\|_{H^s}^2+\|\Lambda^\alpha \mathbf{u}\|_{H^{s+1}}^2)d\tau.
\end{array}
\end{equation}
Note that
\begin{equation}\label{3.30}
\begin{array}{ll}
&\displaystyle\!\!\!\!\|\Delta_{-1}\rho\|_{H^s}^2+\|\Delta_{-1}\mathbf{u}\|_{H^s}^2+\int_0^t(\|\Delta_{-1}\nabla\rho\|^2+\|\Delta_{-1}\Lambda^\alpha \mathbf{u}\|_{H^1}^2)d\tau\\[2mm]
\lesssim&\displaystyle\!\!\!\!\|\rho\|_{H^1}^2+\|\mathbf{u}\|_{H^1}^2+\int_0^t(\|\nabla\rho\|^2+\|\Lambda^\alpha \mathbf{u}\|_{H^1}^2)d\tau.
\end{array}
\end{equation}
Now, multiplying (\ref{3.25}) by $2^{2j(s-1)}$ and then summing for $j\geq 0$, it follows from (\ref{3.29}) and (\ref{3.30}) that
\begin{equation}\label{3.31}
%
\end{equation}
By using Lemma \ref{2 2.8}, we obtain
\begin{equation}\label{3.32}
\begin{array}{ll}
&\!\!\!\!\|\mathbf{u}\cdot\nabla\rho\|_{\dot H^{{s-1}}}^2+\|\rho{\rm div}\mathbf{u}\|_{\dot H^{{s-1}}}^2+\|\mathbf{u}\cdot\nabla \mathbf{u}\|_{\dot H^{{s-1}}}^2+\|\rho\nabla\rho\|_{\dot H^{{s-1}}}^2\\[2mm]
&\!\!\!\!+\|\nabla(\mathbf{u}\cdot\nabla\rho)\|_{\dot H^{{s-1}}}^2+\|\nabla(\rho{\rm div}\mathbf{u})\|_{\dot H^{{s-1}}}^2+\|\nabla(\mathbf{u}\cdot\nabla \mathbf{u})\|_{\dot H^{{s-1}}}^2+\|\nabla(\rho\nabla\rho)\|_{\dot H^{{s-1}}}^2\\[2mm]
\lesssim&\!\!\!\!\|\mathbf{u}\cdot\nabla\rho\|_{H^s}^2+\|\rho{\rm div}u\|_{H^s}^2+\|\mathbf{u}\cdot\nabla \mathbf{u}\|_{H^s}^2+\|\rho\nabla\rho\|_{H^s}^2\\[2mm]
\lesssim&\!\!\!\!\|\mathbf{u}\|_{H^s}^2\|\nabla\rho\|_{H^s}^2+\|\rho\|_{H^s}^2\|{\rm div}\mathbf{u}\|_{H^s}^2+\|\mathbf{u}\|_{H^s}^2\|\nabla \mathbf{u}\|_{H^s}^2+\|\rho\|_{H^s}^2\|\nabla\rho\|_{H^s}^2\\[2mm]
\lesssim&\!\!\!\!(\|\mathbf{u}\|_{H^s}^2+\|\rho\|_{H^s}^2)(\|\nabla\rho\|_{H^s}^2+\|{\rm div}\mathbf{u}\|_{H^s}^2+\|\nabla \mathbf{u}\|_{H^s}^2).
\end{array}%
\end{equation}
Inserting inequalities \eqref{3.32}-\eqref{3.222} into (\ref{3.31}), we immediately gives
\begin{equation}\label{3.33}
\begin{array}{ll}
&\displaystyle\|\rho\|_{H^{s+1}}+\|\mathbf{u}\|_{H^{s+1}}+\int_0^t\|\nabla\rho\|_{H^s}^2+\|\Lambda^\alpha \mathbf{u}\|_{H^{s+1}}^2d\tau\\[2mm]
\lesssim&\displaystyle\|\rho_0\|_{H^{s+1}}+\|\mathbf{u}_0\|_{H^{s+1}}+\int_0^t(\|\rho\|_{H^{s+1}}^2+\|\mathbf{u}\|_{H^{s+1}}^2)(\|\nabla\rho\|_{H^s}^2+\|\Lambda^\alpha \mathbf{u}\|_{H^{s+1}}^2)d\tau,
\end{array}
\end{equation}
where we have used $H^r\hookrightarrow L^\infty(r>\frac{d}{2})$.
\end{proof}
\section{Proof of Theorem 1.1}
In this section, we shall use four steps to prove the existence and uniqueness of (\ref{1.1}) with the small initial data.

{\bf Step 1: Construction of the approximate solutions}. We first construct the approximate solutions: as in \cite{LYZ}, this is based heavily on the classical Friedrich's method. Defining the smoothing operator
$$\mathcal{J}_\varepsilon f=\mathcal{F}^{-1}(1_{0\leq|\xi|\leq\frac{1}{\varepsilon}}\mathcal{F}f),$$
we consider the approximate system of (\ref{2.1}):
\begin{equation}\label{5.1}
\frac{\partial U_\varepsilon}{\partial t}=F_\varepsilon(U_\varepsilon),\ U_\varepsilon=(\rho_\varepsilon,\mathbf{u}_\varepsilon).
\end{equation}
$F_\varepsilon(U_\varepsilon)=(F_\varepsilon^{(1)}(U_\varepsilon),F_\varepsilon^{(2)}(U_\varepsilon))$ be defined by
\begin{equation}\label{5.2}
\left\{
\begin{array}{lc}
F_\varepsilon^{(1)}(U_\varepsilon)=-\kappa{\rm div}(\mathcal{J}_\varepsilon \mathbf{u}_\varepsilon)-\mathcal{J}_\varepsilon(\mathcal{J}_\varepsilon \mathbf{u}_\varepsilon\cdot\nabla \mathcal{J}_\varepsilon\rho_\varepsilon)-\frac{1}{a}\mathcal{J}_\varepsilon(\mathcal{J}_\varepsilon \rho_\varepsilon{\rm div} \mathcal{J}_\varepsilon \mathbf{u}_\varepsilon),\\[2mm]
F_\varepsilon^{(2)}(U_\varepsilon)=-\kappa\nabla\mathcal{J}_\varepsilon\rho_\varepsilon-\mu\Lambda^{2\alpha}\mathcal{J}_\varepsilon \mathbf{u}_\varepsilon-\mathcal{J}_\varepsilon(\mathcal{J}_\varepsilon(\frac{\mu'}{(\kappa+\frac{1}{a})^a}-\mu)_\varepsilon \Lambda^{2\alpha}\mathcal{J}_\varepsilon \mathbf{u}_\varepsilon)\\[2mm]
\ \ \ \ \ \ \ \ \ \ \ \ \ \ \ -J_\varepsilon(\mathcal{J}_\varepsilon \mathbf{u}_\varepsilon\cdot\nabla \mathcal{J}\varepsilon \mathbf{u}_\varepsilon)-\frac{1}{a}\mathcal{J}_\varepsilon(\mathcal{J}_\varepsilon\rho_\varepsilon\nabla \mathcal{J}\varepsilon\rho_\varepsilon).
\end{array}
\right.
\end{equation}
Using the fact that$\|\mathcal{J}_\varepsilon f\|_{H^k}\leq C(1+\frac{1}{\varepsilon^2})^{\frac{k}{2}}\|f\|_{L^2}$, it is easy to obtain,
$$\|F_\varepsilon(U_\varepsilon)\|_{L^2}\leq C_\varepsilon f(\|U_\varepsilon\|_{L^2}),$$
$$\|F_\varepsilon(U_\varepsilon)-F_\varepsilon(\tilde{U}_\varepsilon)\|_{L^2}\leq C_\varepsilon g(\|U_\varepsilon\|_{L^2},\|\tilde{U}_\varepsilon\|_{L^2})\|U_\varepsilon-\tilde{U}_\varepsilon\|_{L^2},$$
where $f$ and $g$ are polynomials with positive coefficients. Therefore, the approximate system can be viewed as an ODE system on $L^2$. By using Cauchy-Lipschitz theorem, we know that there exists a maximal time $T_\varepsilon>0$ and a unique $(\rho_\varepsilon,\mathbf{u}_\varepsilon)$ which is continuous in time with a value in $L^2$. $\mathcal{J}_\varepsilon^2=\mathcal{J}_\varepsilon$ ensures that $(\mathcal{J}_\varepsilon\rho_\varepsilon,\mathcal{J}_\varepsilon \mathbf{u}_\varepsilon)$ is also a solution of (\ref{5.1}). Thus, $(\rho_\varepsilon,\mathbf{u}_\varepsilon)=(\mathcal{J}_\varepsilon\rho_\varepsilon,\mathcal{J}_\varepsilon \mathbf{u}_\varepsilon)$. Then, $(\rho_\varepsilon,\mathbf{u}_\varepsilon)$ satisfies:
\begin{equation}\label{5.3}
\left\{
\begin{array}{lc}
\partial_t\rho_\varepsilon+\kappa{\rm div}(\mathbf{u}_\varepsilon)=-\mathcal{J}_\varepsilon( \mathbf{u}_\varepsilon\cdot\nabla \rho_\varepsilon)-\frac{1}{a}\mathcal{J}_\varepsilon(\rho_\varepsilon{\rm div} \mathbf{u}_\varepsilon),\\[2mm]
\partial_t\mathbf{u}_\varepsilon+\kappa\nabla\rho_\varepsilon-\mu\Lambda^{2\alpha}\mathbf{u}_\varepsilon-\mathcal{J}_\varepsilon((\frac{\mu'}{(\kappa+\frac{1}{a})^a}-\mu)_\varepsilon \Lambda^{2\alpha}\mathbf{u}_\varepsilon)-\mathcal{J}_\varepsilon(\mathbf{u}_\varepsilon\cdot\nabla \mathbf{u}_\varepsilon)-\frac{1}{a}J_\varepsilon(\rho_\varepsilon\nabla\rho_\varepsilon).
\end{array}
\right.
\end{equation}
Moreover, we can also conclude that $(\rho_\varepsilon, \mathbf{u}_\varepsilon)\in E(T_\epsilon)$.

{\bf Step 2: Uniform energy estimates}. Now we will prove that $\|(\rho_\varepsilon,\mathbf{u}_\varepsilon)\|_{E(T)}$ is uniformly bounded independently of $\varepsilon$ by losing energy estimate.
Suppose that $\eta$ is small such that $\|\rho_0\|_{L^\infty}\leq\frac{1}{4}$. Since the solution depends continuously on the time variable, then there exist $0<T_0<T_\varepsilon$ and a positive constant $\mathfrak{C}$ such that the solution $(\rho_\varepsilon,u_\varepsilon)$ satisfies that
$$
\|\rho_\varepsilon(t,\cdot)\|_{L^\infty}\leq\frac{1}{2}\ for\ all\ t\in[0,T_0],
$$
$$
\|(\rho_\varepsilon,\mathbf{u}_\varepsilon)\|_{E(T_0)}\leq2\mathfrak{C}\|(\rho,\mathbf{u})\|_{E(0)}.
$$
We suppose that $T_0$ is a maximal time so that the above inequalities hold without loss of generality. Next, we will get a refined estimate on $[0, T_0]$ for the solution. According to Proposition \ref{3 3.1}, for all $0<T\leq T_0$, one has
\begin{equation}\label{5.4}
\begin{array}{ll}
\|(\rho_\varepsilon,\mathbf{u}_\varepsilon)\|_{E(T)}&\leq\mathfrak{C}\|(\rho,\mathbf{u})\|_{E(0)}+\mathfrak{C}\|(\rho_\varepsilon,\mathbf{u}_\varepsilon)\|_{E(T)}^2\leq\mathfrak{C}\|(\rho,\mathbf{u})\|_{E(0)}(1+4\mathfrak{C}^2\eta).
\end{array}
\end{equation}
Let $\eta<\frac{1}{8\mathfrak{C}^2}$, then
\begin{equation}\label{5.5}
\begin{array}{ll}
\|(\rho_\varepsilon,\mathbf{u}_\varepsilon)\|_{X(T_0)}\leq\frac{3}{2}\mathfrak{C}\|(\rho,\mathbf{u})\|_{E(0)}<2\mathfrak{C}\eta.
\end{array}
\end{equation}
A standard bootstrap argument yields for all $0<T<\infty$ that
$$
\|(\rho_\varepsilon,\mathbf{u}_\varepsilon)\|_{E(T)}\leq 2\mathfrak{C}\|(\rho,\mathbf{u})\|_{E(0)}.
$$
{\bf Step 3: Existence of the solution}. $(\rho,\mathbf{u})\in E(T)$ of (\ref{2.1}) can be deduced by a standard compactness argument to the approximation sequence $(\rho_\varepsilon,u_\varepsilon)$; for convenience, we omit here. Moreover, it holds for all $0<T<\infty$ that
$$
\|(\rho,\mathbf{u})\|_{E(T)}\leq 2\mathfrak{C}\|(\rho,\mathbf{u})\|_{E(0)}.
$$
{\bf Step 4: Uniqueness of the solution}. The uniqueness of the solution can be guaranteed by Step 3 as in \cite{LYZ}, thus we omit here.

\section{Optimal decay rates}
\subsection{The energy estimates on the highest-order derivative}
In this section, based on the energy method together with low-high frequency decompositon, we get the estimates on the ${\sigma_0}$-order (${\sigma_0}>\frac{5}{2}$) derivative of solution $(\rho,\mathbf{u})$ to the Cauchy problem (\ref{2.1})-(\ref{2.2}).

\begin{lemma}\label{4 4.1}
There exist suitably small constants $\beta_3>0$ and $\epsilon_1>0$ such that
\begin{equation}\label{4.1}

\end{equation}
Here $\eta>0$ is defined in (\ref{1.3}) and sufficiently small.
\end{lemma}
\begin{proof}
Multiplying $\Lambda^{\sigma_0}(\ref{2.1})_1-\Lambda^{\sigma_0}(\ref{2.1})_2$ by $\Lambda^{\sigma_0}\rho$ and $\Lambda^{\sigma_0} \mathbf{u}$, respectively, integrating the resultant inequality with respect to $x$ over $\mathbb{R}^3$ and using Young's inequality, we get
\begin{equation*}\label{4.200}
%
\end{equation}
By using Lemma \ref{2 2.6} and Young's inequality, we get
\begin{equation}\label{4.3}
%
\end{equation}
Similarly, we have
\begin{equation}\label{4.4}
%
\end{equation}
where we have used $H^{r}(\mathbb{R}^d)\hookrightarrow L^\infty(\mathbb{R}^d)$ with $r>\frac{d}{2}$ and Young's inequality.

Multiplying $\Lambda^{{\sigma_0}-1}(\ref{2.1})_2$ by $\nabla\Lambda^{{\sigma_0}-1}\rho$ and using $(\ref{2.1})_1$, we have
\begin{equation}\label{4.6}
%
\end{equation}
By virtue of (\ref{2.13}) and Young's inequality, we get
\begin{equation}\label{4.7}
%
\end{equation}
In the same way, we give
\end{equation}
Summing up (\ref{4.5}) and $\beta_3(\ref{4.9})$ with suitably small $\beta_3>0$, we obtain
\begin{equation}\label{4.10}
%
\end{equation}
This completes the proof.
\end{proof}

\subsection{Cancellation of a low-frequency part}
\begin{lemma}\label{4 2.1}
It holds that
\begin{equation}\label{4.01}
\begin{array}{ll}
\displaystyle\|(\Lambda^{\sigma_0}\rho,\Lambda^{\sigma_0} \mathbf{u})\|^2\lesssim e^{-C_2t}\|(\Lambda^{\sigma_0}\rho_0,\Lambda^{\sigma_0} \mathbf{u}_0)\|^2+\int_0^te^{-C_2(t-\tau)}\|(\Lambda^{\sigma_0}\rho^L,\Lambda^{\sigma_0} \mathbf{u}^L)\|^2d\tau,
\end{array}
\end{equation}
where the positive constant $C_2$ is independent of $\eta$.
\end{lemma}
\begin{proof}
Multiplying $\Lambda^{{\sigma_0}-1}(\ref{2.1})_2$ by $\nabla\Lambda^{{\sigma_0}-1}\rho^L$ and using $(\ref{2.1})_1$, we have
\begin{equation}\label{4.11}
%
\end{equation}
Adding (\ref{4.10}) and $\beta_3(\ref{4.14})$ with suitably small $\beta_3>0$, we obtain
\begin{equation}\label{4.15}
%
\end{equation}
Then, there exists a constant $C_2>0$ such that
\begin{equation}\label{4.16}
%
\end{equation}
Multiplying (\ref{4.16}) by $e^{C_2t}$ and integrating with respect to $t$ over $[0,t]$, we have
\begin{equation}\label{4.17}
%
\end{equation}
This proves Lemma \ref{4 2.1}.
\end{proof}
\subsection{Decay rates for the nonlinear system}
Based on spectral analysis in \cite{Wangs1}, we have the following lemma.
\begin{lemma}\label{4 3.1}
Suppose that $U=(\rho,\mathbf{u})$ is the solution of the Cauchy problem of the following nonlinear problem:
\begin{equation}\label{4.18}
\left\{
%
\right.
\end{equation}
with the initial data $U_0=(\rho_0,\mathbf{u}_0)$. Then for any integer $j\geq0$, it holds that
\begin{equation}\label{4.19}
%
\end{equation}
where $F(U)=(F_1(U),F_2(U))$.
\end{lemma}

In this subsection, by combining Lemma \ref{4 4.1} with Lemma \ref{4 2.1} and Lemma \ref{4 3.1}, we get the time-decay rates of the solution to the nonlinear Cauchy problem of (\ref{2.1}).
\begin{lemma}\label{4 3.2}
Under the assumption of Theorem \ref{1 1.2}, it holds that
\begin{equation}\label{4.20}
\|\Lambda^\sigma(\rho,\mathbf{u})(t)\|_{L^2(\mathbb{R}^3)}\lesssim (1+t)^{-\frac{3}{4\alpha}-\frac{\sigma}{2}},\ \ 0\leq \sigma\leq {\sigma_0}.
\end{equation}
\end{lemma}
\begin{proof}
Denote that
$$M(t):=\sup\limits_{0\leq\tau\leq t}\sum\limits_{m=0}^{\sigma_0}(1+\tau)^{\frac{3}{4\alpha}+\frac{m}{2}}\|\Lambda^m(\rho,\mathbf{u})\|.$$
Notice that $M(t)$ is non-decreasing, then for $0\leq m\leq {\sigma_0}$,
$$\|\Lambda^m(\rho,\mathbf{u})(\tau)\|\leq C_3(1+\tau)^{-\frac{3}{4\alpha}-\frac{m}{2}}M(t),\ 0\leq\tau\leq t$$
holds true for some positive constant $C_3$ independent of $\eta$.
By using H\"{o}lder's inequality, we have
\begin{equation}\label{4.21}
\begin{array}{ll}
\|F(U)(\tau)\|_{L^1}&\!\!\!\!\displaystyle\lesssim\|\mathbf{u}\|(\|\nabla\rho\|+\|\nabla \mathbf{u}\|)+\|\rho\|(\|{\rm div}\mathbf{u}\|+\|\nabla\rho\|)+\|\rho\|\|\Lambda^{2\alpha}\mathbf{u}\|\\[2mm]
&\!\!\!\!\displaystyle\lesssim \eta M(t)(1+\tau)^{-\frac{3}{4\alpha}-\frac{1}{2}},
\end{array}
\end{equation}
and
\begin{equation}\label{4.22}
\begin{array}{ll}
\|F(U)(\tau)\|&\!\!\!\!\displaystyle\lesssim\|\mathbf{u}\|_{L^3}(\|\nabla\rho\|_{L^6}+\|\nabla \mathbf{u}\|_{L^6})+\|\rho\|_{L^3}(\|{\rm div}\mathbf{u}\|_{L^6}+\|\nabla\rho\|_{L^6})\\[2mm]
&\!\!\!\!\displaystyle\quad+\|\rho\|_{L^{\frac{6}{2{\sigma_0}-4\alpha}}}\|\Lambda^{2\alpha}\mathbf{u}\|_{L^{\frac{6}{3-2{\sigma_0}+4\alpha}}}\\[2mm]
&\!\!\!\!\displaystyle\lesssim\|\mathbf{u}\|_{H^1}(\|\nabla\nabla\rho\|+\|\nabla\nabla \mathbf{u}\|)+\|\rho\|_{H^1}(\|\nabla{\rm div}\mathbf{u}\|+\|\nabla\nabla\rho\|)\\[2mm]
&\!\!\!\!\displaystyle\quad+\|\Lambda^{\frac{3+4\alpha-2{\sigma_0}}{2}}\rho\|\|\Lambda^{\sigma_0} \mathbf{u}\|\\[2mm]
&\!\!\!\!\displaystyle\lesssim \eta^{1-\epsilon_2} M(t)^{1+\epsilon_2}(t)(1+\tau)^{-\frac{3}{4\alpha}-1-\frac{3}{4\alpha}\epsilon_2},
\end{array}
\end{equation}
where $\epsilon_2\in(0,\frac{1}{2})$. By [15, Lemma 2.5], Lemma \ref{4 3.1}, (\ref{4.21}) and (\ref{4.22}), we have for $0\leq \sigma\leq {\sigma_0}$,
\begin{equation}\label{4.23}
\begin{array}{ll}
\|(\Lambda^\sigma\rho^L,\Lambda^\sigma\mathbf{u}^L)(t)\|&\!\!\!\!\displaystyle\lesssim(1+t)^{-\frac{3}{4\alpha}-\frac{\sigma}{2}}\|(\rho_0,\mathbf{u}_0)\|_{L^1(\mathbb{R}^3)}\\
&\!\!\!\!\displaystyle\quad\displaystyle+\eta M(t)\int_0^{\frac{t}{2}}(1+t-\tau)^{-\frac{3}{4\alpha}-\frac{\sigma}{2}}(1+\tau)^{-\frac{3}{4\alpha}-\frac{1}{2}}d\tau\\[2mm]
&\!\!\!\!\displaystyle\quad+\eta^{1-\epsilon_2} M(t)^{1+\epsilon_2}(t)\int_{\frac{t}{2}}^t(1\!+\!t\!-\!\tau)^{-\frac{\sigma}{2}}(1+\tau)^{-\frac{3}{4\alpha}-1-\frac{3}{4\alpha}\epsilon_2}d\tau\\[2mm]
&\!\!\!\!\displaystyle\lesssim\big(\|(\rho_0,\mathbf{u}_0)\|_{L^1(\mathbb{R}^3)}+\eta M(t)+\eta^{1-\epsilon_2} M(t)^{1+\epsilon_2}(t)\big)(1+t)^{-\frac{3}{4\alpha}-\frac{\sigma}{2}}.
\end{array}
\end{equation}
It follows from (\ref{4.17}) that
\begin{equation}\label{4.24}
\begin{array}{ll}
\|(\Lambda^{\sigma_0}\rho,\Lambda^{\sigma_0}\mathbf{u})(t)\|^2&\!\!\!\!\displaystyle\lesssim\displaystyle e^{-C_2t}(\|(\Lambda^{\sigma_0}\rho_0,\Lambda^{\sigma_0}\mathbf{u}_0)\|^2)+\int_0^te^{-C_2(t-\tau)}\big(\|(\rho_0,\mathbf{u}_0)\|_{L^1(\mathbb{R}^3)}^2\\[2mm]
&\!\!\!\!\displaystyle\quad+\eta^2 M^2(t)+\eta^{2-2\epsilon_2} M(t)^{2+2\epsilon_2}(t)\big)(1+\tau)^{-\frac{3}{2\alpha}-{\sigma_0}}d\tau\\[2mm]
&\!\!\!\!\displaystyle\lesssim\big(\|(\rho_0,\mathbf{u}_0)\|_{H^{\sigma_0}\cap L^1}+\eta^2 M^2(t)+\eta^{2-2\epsilon_2} M(t)^{2+2\epsilon_2}(t)\big)(1+\tau)^{-\frac{3}{2\alpha}-{\sigma_0}}.
\end{array}
\end{equation}
Moreover, using the decomposition (\ref{6.3}) gives for $0\leq \sigma\leq {\sigma_0}$ that
\begin{equation}\label{4.25}
\begin{array}{ll}
\|(\Lambda^\sigma\rho,\Lambda^\sigma\mathbf{u})(t)\|^2&\!\!\!\!\displaystyle\lesssim\|(\Lambda^\sigma\rho,\Lambda^\sigma\mathbf{u})^L(t)\|^2+\|(\Lambda^\sigma\rho,\Lambda^\sigma\mathbf{u})^H(t)\|^2\\[2mm]
&\!\!\!\!\displaystyle\lesssim\|(\Lambda^\sigma\rho,\Lambda^\sigma\mathbf{u})^L(t)\|^2+\|(\Lambda^{\sigma_0}\rho,\Lambda^{\sigma_0}\mathbf{u})(t)\|^2.
\end{array}
\end{equation}
From (\ref{4.23}), (\ref{4.24}) and (\ref{4.25}), for $0\leq \sigma\leq {\sigma_0}$, we have
\begin{equation}\label{4.26}
\begin{array}{ll}
\|(\Lambda^\sigma\rho,\Lambda^\sigma\mathbf{u})(t)\|^2\lesssim\big(\|(\rho_0,\mathbf{u}_0)\|_{H^{\sigma_0}\cap L^1}+\eta^2 M^2(t)+\eta^{2-2\epsilon_1} M(t)^{2+2\epsilon_2}(t)\big)(1+\tau)^{-\frac{3}{2\alpha}-\sigma}.
\end{array}
\end{equation}
By noting the definition of $M(t)$ and using the smallness of $\eta$, from (\ref{4.26}), there exists a positive constant $C_4$ independent of $\eta$ such that
\begin{equation}\label{4.27}
\begin{array}{ll}
M^2(t)\leq C_4\big(\|(\rho_0,\mathbf{u}_0)\|_{H^{\sigma_0}\cap L^1}+\eta^2 M^2(t)+\eta^{2-2\epsilon_2} M(t)^{2+2\epsilon_2}(t)\big).
\end{array}
\end{equation}
By using Young's inequality, we obtain
\begin{equation}\label{4.28}
C_4\eta^{2-2\epsilon_2} M(t)^{2+2\epsilon_2}(t)\leq\frac{1-\epsilon_2}{2}C_4^{\frac{2}{1-\epsilon_2}}+\frac{1+\epsilon_2}{2}\eta^{\frac{4(1-\epsilon_2)}{1+\epsilon_2}}M^4(t).
\end{equation}
For simplicity, we denote
\begin{equation}\label{4.29}
K_0:=C_4\|(\rho_0,\mathbf{u}_0)\|_{H^{\sigma_0}\cap L^1}+\frac{1-\epsilon_2}{2}C_4^{\frac{2}{1-\epsilon_2}},
\end{equation}
and
\begin{equation}\label{4.30}
C_\eta:=\frac{1+\epsilon_2}{2}\eta^{\frac{4(1-\epsilon_2)}{1+\epsilon_2}}.
\end{equation}
From (\ref{4.27}) and the smallness of $\eta$, we have
\begin{equation}\label{4.31}
M^2(t)\leq K_0+C_4M^4(t).
\end{equation}
Notice that $M(t)$ is non-decreasing and continuous, we have $M(t)\leq C$ for any $t\in[0,+\infty)$. This implies
$$
\|\Lambda^\sigma(\rho,\mathbf{u})(t)\|_{L^2(\mathbb{R}^3)}\lesssim (1+t)^{-\frac{3}{4\alpha}-\frac{\sigma}{2}},\ 0\leq\sigma\leq\sigma_0,
$$
and completes the proof.
\end{proof}
\section{Appendix}
\quad\quad In order to obtain the optimal decay rates of the solution, we need the content in Appendix.
First, we need the frequency decomposition of the solution:
\begin{equation*}\label{6.1}
f^l(x)=\chi_0(\partial_x)f(x),\ f^h(x)=\chi_1(\partial_x)f(x),\ f^m(x)=(1-\chi_0(\partial_x)-\chi_1(\partial_x))f(x).
\end{equation*}
Here $\chi_0(\partial_x)=\mathcal{F}^{-1}(\chi_0(\xi))$ and $\chi_1(\partial_x)=\mathcal{F}^{-1}(\chi_1(\xi))$ satisfy
$0\leq\chi_0(\xi),\chi_1(\xi)\leq1$,
and
\begin{equation*}\label{6.2}
\chi_0(\xi)=\left\{\begin{array}{rl}
1,\ |\xi|<\frac{r_0}{2},\\
0,\ |\xi|>r_0,
\end{array}\right.
\ \ \chi_1(\xi)=\left\{\begin{array}{ll}
0,\ |\xi|<R_0,\\
1,\ |\xi|>R_0+1,
\end{array}\right.
\end{equation*}
for some fixed $r_0$ and $R_0$. Therefore, it is easy to see that
\begin{equation}\label{6.3}
f(x)=f^l(x)+f^m(x)+f^h(x)=:f^L(x)+f^h(x)=:f^l(x)+f^H(x),
\end{equation}
where $f^L(x)=f^l(x)+f^m(x)$ and $f^H(x)=f^m(x)+f^h(x)$, and have the following inequalities:

\begin{lemma}\label{l 6.1}\cite{WW} For $f(x)\in H^s(\mathbb{R}^3)$ and any given integers $k,k_0,k_1$ with $k_0\leq k\leq k_1\leq s$, it holds that
\begin{equation*}
\begin{array}{rl}
&\|\partial_x^kf^l\|_{L^2(\mathbb{R}^3)}\leq r_0^{k-k_0}\|\partial_x^{k_0}f^l\|_{L^2(\mathbb{R}^3)},\ \ \ \|\partial_x^kf^l\|_{L^2(\mathbb{R}^3)}\leq \|\partial_x^{k_1}f\|_{L^2(\mathbb{R}^3)},\\[2mm]
&\|\partial_x^kf^h\|_{L^2(\mathbb{R}^3)}\leq \frac{1}{R_0^{k_1-k}}\|\partial_x^{k_1}f^h\|_{L^2(\mathbb{R}^3)},\ \ \ \|\partial_x^kf^h\|_{L^2(\mathbb{R}^3)}\leq \|\partial_x^{k_1}f\|_{L^2(\mathbb{R}^3)},\\[2mm]
&\ \ \ \ \ r_0^k\|f^m\|_{L^2(\mathbb{R}^3)}\leq\|\partial_x^kf^m\|_{L^2(\mathbb{R}^3)}\leq R_0^k\|f^m\|_{L^2(\mathbb{R}^3)}.
\end{array}
\end{equation*}
\end{lemma}

\bigbreak\bigbreak
\noindent{\bf Funding}: The research was supported by National Natural Science Foundation of China (No. 11971100), Natural Science Foundation of Shanghai (Grant No. 22ZR1402300) and the Fundamental Research Funds for the Central Universities
(No. 2232023A-02 and 2232024G-13).\\

\bibliographystyle{plain}

\end{document}